\documentclass{amsart}

\usepackage{pgf,tikz,pgfplots}
\pgfplotsset{compat=1.15}
\usepackage{mathrsfs}
\usetikzlibrary{arrows}

\usepackage[T1]{fontenc}
\usepackage{microtype}
\usepackage{lmodern}
\usepackage[colorlinks=true,urlcolor=blue, citecolor=red,linkcolor=blue,linktocpage,pdfpagelabels, bookmarksnumbered,bookmarksopen]{hyperref}
\usepackage[hyperpageref]{backref}
\usepackage{amsthm} 
\usepackage{latexsym,amsmath,amssymb}

\usepackage{accents}
\usepackage{esint}

\usepackage{soul}
\usepackage{mathtools} 
\usepackage{xparse} 
\usepackage[capitalize]{cleveref}

\usepackage[shortlabels]{enumitem}

\usepackage[nomargin,inline,marginclue,draft]{fixme}
\FXRegisterAuthor{n}{nlang}{\color{red} Nicole}
\fxusetheme{color}
\usepackage{layout} 
\usepackage[top=3cm, bottom=3cm, left=2cm, right=2cm]{geometry} 

\usepackage[textsize=small]{todonotes}
\newcommand{\hdg}{\star}

\title[$n$-Laplace systems with antisymmetric potentials]{Regularizing properties of $n$-Laplace systems with antisymmetric potentials in Lorentz spaces}
\author{Dorian Martino}
\address[Dorian Martino]{Institut de Mathématiques de Jussieu, Université Paris Cité, Bâtiment
	Sophie Germain, 75205 Paris Cedex 13, France}
\email{dorian.martino@imj-prg.fr}

\author{Armin Schikorra}
\address[Armin Schikorra]{Department of Mathematics,
	University of Pittsburgh,
	301 Thackeray Hall,
	Pittsburgh, PA 15260, USA}
\email{armin@pitt.edu}

\newcommand{\N}{{\mathbb N}}

\renewcommand{\S}{{\mathbb S}}

\newtheorem{theorem}{Theorem}
\newtheorem{lemma}[theorem]{Lemma}
\newtheorem{corollary}[theorem]{Corollary}
\newtheorem{proposition}[theorem]{Proposition}

\theoremstyle{definition}

\newtheorem{question}[theorem]{Question}

\theoremstyle{remark}

\newcommand\osc{\mathop{\rm osc\,}}

\newcommand\curl{{\rm curl\,}}

\newcommand\supp{{\rm supp\,}}

\newcommand{\g}{\nabla}
\newcommand{\dr}{\partial}
\newcommand{\di}{\mathrm{div }}

\newcommand{\R}{\mathbb{R}}

\newcommand{\Nr}{\mathcal{N}}

\newcommand{\brac}[1]{\left (#1 \right )}
\newcommand{\abs}[1]{\left\lvert #1 \right \rvert}

\newcommand{\Ep}{\bigwedge\nolimits}

\newcommand{\scal}[2]{\left\langle #1,#2 \right\rangle}
\newcommand{\inter}[2]{[\![#1,#2]\!]}

\newcommand{\barint}{
	\rule[.036in]{.12in}{.009in}\kern-.16in \displaystyle\int }

\newcommand{\barcal}{\text{$ \rule[.036in]{.11in}{.007in}\kern-.128in\int $}}

\def\mvint_#1{\mathchoice
	{\mathop{\vrule width 6pt height 3 pt depth -2.5pt
			\kern -8pt \intop}\nolimits_{\kern -3pt #1}}%
{\mathop{\vrule width 5pt height 3 pt depth -2.6pt
		\kern -6pt \intop}\nolimits_{#1}}%
{\mathop{\vrule width 5pt height 3 pt depth -2.6pt
		\kern -6pt \intop}\nolimits_{#1}}%
{\mathop{\vrule width 5pt height 3 pt depth -2.6pt
		\kern -6pt \intop}\nolimits_{#1}}}

		\numberwithin{theorem}{section} \numberwithin{equation}{section}

\newcommand{\lap}{\Delta }
\newcommand{\aleq}{\lesssim}

\newcommand{\aeq}{\approx}

\newcommand{\Rz}{\mathcal{R}}
\newcommand{\laps}[1]{|D|^{#1}}

\newcommand{\lapms}[1]{\mathcal{I}_{#1}}

\usepackage{scalerel}[2014/03/10]
\usepackage[usestackEOL]{stackengine}
\def\avint{\,\ThisStyle{\ensurestackMath{%
		\stackinset{c}{.2\LMpt}{c}{.5\LMpt}{\SavedStyle-}{\SavedStyle\phantom{\int}}}%
	\setbox0=\hbox{$\SavedStyle\int\,$}\kern-\wd0}\int}
	
	\renewcommand{\div}{\operatorname{div}}
	
	\let\latexchi\chi
	\makeatletter
	\renewcommand\chi{\@ifnextchar_\sub@chi\latexchi}
	\newcommand{\sub@chi}[2]{
\@ifnextchar^{\subsup@chi{#2}}{\latexchi^{}_{#2}}%
}
\newcommand{\subsup@chi}[3]{
\latexchi_{#1}^{#3}%
}
\makeatother
\newcommand{\eps}{\varepsilon}
\newcommand{\ve}{\varepsilon}
\newcommand{\vp}{\varphi}

\begin{document}
\begin{abstract}
	We show continuity of solutions $u \in W^{1,n}(B^n,\mathbb{R}^N)$ to the system
	\[
	-{\rm div} (|\nabla u|^{n-2} \nabla u) = \Omega \cdot |\nabla u|^{n-2} \nabla u
	\]
	when $\Omega$ is an $L^n$-antisymmetric potential -- and additionally satisfies a Lorentz-space assumption.
	
	To obtain our result we study a rotated n-Laplace system
	\[
	-{\rm div} (Q|\nabla u|^{n-2} \nabla u) = \tilde{\Omega} \cdot |\nabla u|^{n-2} \nabla u,
	\]
	where $Q \in W^{1,n}(B^n,SO(N))$ is the Coulomb gauge which ensures improved Lorentz-space integrability of $\tilde{\Omega}$. Because of the matrix-term $Q$, this system does not fall directly into Kuusi--Mingione's vectorial potential theory. However, we adapt ideas of their theory together with Iwaniec' stability result to obtain $L^{(n,\infty)}$-estimates of the gradient of a solution which, by an iteration argument leads to the regularity of solutions. As a corollary of our argument we see that $n$-harmonic maps into manifolds are continuous if their gradient belongs to the Lorentz-space $L^{(n,2)}$ -- which is a trivial and optimal assumption if $n=2$, and the weakest assumption to date for the regularity of critical $n$-harmonic maps, without any added differentiability assumption. We also prove a corresponding result for $n$-Laplace $H$-systems.
\end{abstract}

\maketitle
\tableofcontents

\section{Introduction}
Critical harmonic maps between a Euclidean ball $B \subset \R^2$ and a closed manifold without boundary $\mathcal{N}$, i.e. critical points of the Dirichlet energy subject to the previous mapping conditions, satisfy the Euler-Lagrange equation
\[
-\lap u = A(u) (\nabla u,\nabla u) \quad \text{in $B$}.
\]
Here $A(u)$ denotes the second fundamental form of the manifold $\mathcal{N} \subset \R^N$.

Since $u \in W^{1,2}(B,\R^N)$, the right-hand side of this equation is merely integrable, and one might be lead to believe that $u$ could be discontinuous, since the scalar model equation $|\lap u| = |\nabla u|^2$ has discontinuous solution $\log \log 1/|x|$.

However, solutions are continuous -- as was shown in H\'elein's celebrated \cite{Helein1,Helein2}. The reason are cancellation effects of div-curl type. Namely, if $\mathcal{N} = \S^{N-1}$ is the unit sphere, H\'elein showed that with Shatah's conservation law \cite{Shatah88} the right-hand side can be written as
\begin{equation}\label{eq:heleinpde}
	-\lap u^i = \Omega_{ij}\cdot \nabla u \quad \text{in $B$}
\end{equation}
where
\[
\Omega_{ij} = u^j \nabla u^i - u^i \nabla u^j
\]
is divergence free, $\div (\Omega_{ij}) = 0$.
The div-curl term on the right hand side of \eqref{eq:heleinpde} implies via the Coifman-Lions-Meyer-Semmes \cite{CLMS} that $\lap u$ belongs to the Hardy space $\mathcal{H}^1$ and then one concludes by Calder\'on-Zygmund theory that a solution $u$ to \eqref{eq:heleinpde} actually belongs to $W^{2,1}_{loc}(B) \subset C^0(B)$. H\'elein then extended this to the general smooth manifold case in \cite{Helein2}, using a moving frame method to obtain an approximate div-curl structure.
More than fifteen years later, Rivi\`ere showed in his celebrated work \cite{R07} that the harmonic map equation -- and actually Euler-Lagrange equations for a large class of conformally invariant variational functionals -- can be written in the form
\begin{equation}\label{eq:riviere1}
	-\lap u^i = \Omega_{ij}\cdot \nabla u^j,
\end{equation}
where $\Omega_{ij} \in L^2(B,\R^N)$ may not be divergence free -- but  it is antisymmetric,
\[
\Omega_{ij} = -\Omega_{ji}.
\]
He then used Uhlenbeck's Coulomb gauge \cite{U82}, namely a rotation $Q \in W^{1,2}(B,SO(N))$, to transform the equation \eqref{eq:riviere1} into the form
\begin{equation}\label{eq:riviere2}
	-\div (Q \nabla u) = \Omega_Q Q \nabla u
\end{equation}
where
\[
\Omega_Q = Q \nabla Q^T + Q \Omega Q^T
\]
is divergence free. This is good, because the (approximate) div-curl structure on the right-hand side of \eqref{eq:riviere2} is enough to obtain continuity of solutions, cf. \cite{RS08}.

Since the first results of H\'elein \cite{Helein1,Helein2}, attempts were made to generalize this to $n$-Laplace situation. This is a natural question since for $n$-dimensional domains $B \subset \R^n$ the energy $\int_{B} |\nabla u|^n$ is a conformally invariant variational functional. The Euler-Lagrange equations change only slightly,
\begin{equation}\label{eq:nharmmap}
	-\div(|\nabla u|^{n-2} \nabla u) = |\nabla u|^{n-2} A(u) (\nabla u,\nabla u) \quad \text{in $B$}.
\end{equation}
And indeed, for the sphere case $\mathcal{N} = \S^{N-1}$ it is possible to extend H\'elein's approach which was observed by several authors independently \cite{F93,S94,MY96}.
But already for general manifolds the regularity theory is wide open until today. In \cite[III.23]{Riv11} Rivi\`ere asked the following question about regularizing effects of antisymmetric potentials for systems.
\begin{question}[Rivi\`ere]
	Is the following true or false? If for $B \subset \R^n$ we have a solution $u \in W^{1,n}(B,\R^N)$ to
	\[
	-\div(|\nabla u|^{n-2} \nabla u^i) = |\nabla u|^{n-2} \Omega_{ij} \nabla u^j \quad \text{in $B$}
	\]
	where
	\begin{equation}\label{eq:Omegaantisym}
		\Omega_{ij} = -\Omega_{ji} \in L^{n}(B,\R^{n}), \quad i,j \in \{1,\ldots,N\}
	\end{equation}
	then $u$ is continuous.
\end{question}

The underlying reason that this is a very difficult question is that we do not know how to deal with an \emph{approximate} div-curl system, i.e. when the right-hand side is a div-curl term times an $L^\infty \cap W^{1,n}$-map. Also the Hardy-space looses its relevance, \cite{F95}. See \cite{SS17} for a survey over attempts and description of the problems encountered.

There have been several partial results regarding the above question, and we refer again to \cite{SS17} and the recent \cite{MPS22} for an overview. What these results have in common is that they assume additional $W^{1+s,\frac{n}{1+s}}$-regularity for some $s > 0$.

We will not answer the question by Rivi\`ere in this work, but we present a regularity result that does not require any additional differentiability, but only an improved integrability on the zero-order curl of $\Omega$. Denote by $\Rz_\alpha = \partial_\alpha (-\lap)^{-\frac{1}{2}}$ the Riesz transforms (zero-order $\alpha$-derivatives). By $L^{(p,q)}$ we denote the Lorentz space, cf. \Cref{s:lorentzspace}. Then we have the following result

\begin{theorem}\label{th:main1}
	Let $B \subset \R^n$ be a ball. Assume that $u \in W^{1,n}(B,\R^N)$ solves the equation
	\[
	-\div (|\nabla u|^{n-2} \nabla u^i) = \Omega_{ij}\cdot |\nabla u|^{n-2} \nabla u^j \quad \text{in $B$}
	\]
	where we assume antisymmetry, i.e. \eqref{eq:Omegaantisym}, boundedness in $L^{(n,2)}$
	\[
	\|\Omega\|_{L^{(n,2)}(B)} <\infty,
	\]
	and additionally a zero-order curl condition
	\begin{equation}\label{eq:curlcond}
		\max_{i,j = 1,\ldots,N}\max_{\alpha,\beta \in \{1,\ldots,n\}}\|\Rz_\alpha \Omega_{ij}^\beta - \Rz_\beta \Omega_{ij}^\alpha\|_{L^{(n,1)}(B)} < \infty
	\end{equation}
	Then $u$ is continuous in $B$.
\end{theorem}

The antisymmetry assumption \eqref{eq:Omegaantisym} is crucial, since the usual $\log \log 1/|x|$-example can be used to construct counterexamples otherwise, cf. \cite{K10,IO07}.

An immediate corollary for which we have no application whatsoever is the following.
\begin{corollary}\label{co:joke}
	There exists a $\gamma \in (0,1)$ such that the following holds. Let $B \subset \R^n$ be a ball.
	
	Assume that $u \in W^{1,n}(B,\R^N)$ solves the equation
	\[
	\div (|\nabla u|^{n-2} \nabla u^i) = \nabla a_{ij} |\nabla u|^{n-2} \nabla u^j \quad \text{in $B$}
	\]
	where $\nabla a_{ij} \in L^{(n,2)}(B)$ with
	\[
	\nabla a_{ij} = -\nabla a_{ji} \quad i,j \in \{1,\ldots,N\}
	\]
	Then $u$ is continuous in $B$.
\end{corollary}

More relevantly, we observe that by Sobolev embedding \eqref{eq:curlcond} is in particular satisfied if $\|\curl \Omega\|_{L^{\frac{n}{2},1}} < \infty$. From the H\"older inequality for Lorentz-spaces we conclude the following corollary.
\begin{corollary}\label{co:joke2}
	There exists a $\gamma \in (0,1)$ such that the following holds. Let $B \subset \R^n$ be a ball.
	
	Assume that $u \in W^{1,n}(B,\R^N)$ solves the equation
	\[
	\div (|\nabla u|^{n-2} \nabla u^i) = \Omega_{ij}\cdot |\nabla u|^{n-2} \nabla u^j \quad \text{in $B$}
	\]
	where we assume antisymmetry, i.e. \eqref{eq:Omegaantisym}, and additionally for each $i,j$ we assume that
	\[\Omega_{ij} = B_{ij}\, \nabla a_{ij}  \]
	with $B_{ij} \in L^\infty$, $\nabla B_{ij}$, $\nabla a_{ij} \in L^{(n,2)}$.
	
	Then $u$ is continuous in $B$.
\end{corollary}

A particular situation to which the structure of \Cref{co:joke2} applies is the harmonic map equation \eqref{eq:nharmmap}. Namely, we obtain the following
\begin{corollary}\label{co:harmonicmap}
	Let $u \in W^{1,n}(B,\mathcal{N})$ be an $n$-harmonic map from an $n$-dimensional ball $B$ into a smooth manifold without boundary $\mathcal{N} \subset \R^N$, i.e. a solution to \eqref{eq:nharmmap}. If we additionally assume $\nabla u \in L^{(n,2)}(B)$ then $u$ is continuous.
\end{corollary}
By Sobolev embedding theorem, \Cref{co:harmonicmap} contains in particular all previous partial results for harmonic maps that were discussed in \cite{SS17} and the recent \cite[Theorem 1.2]{MPS22} -- without any assumption on additional differentiability. In particular we obtain a positive answer to \cite[Question 1.3 and Question 1.4]{MPS22}.\\

As for the proof of \Cref{th:main1}, using absolute continuity of norms, choosing the Riviere--Uhlenbeck's Coulomb gauge, \Cref{co:reductiontoLn1}, we are able to reduce \Cref{th:main1} to the following
\begin{theorem}\label{th:regularityLn1}
	Assume that $u \in W^{1,n}(B,\R^N)$ solves the equation
	\begin{equation}\label{eq:divQnpde}
		-\div (Q|\nabla u|^{n-2} \nabla u^i) = \Omega_{ij} \cdot |\nabla u|^{n-2} \nabla u^j \quad \text{in $B$}
	\end{equation}
	where we assume $Q \in W^{1,n}(B,SO(N))$ and $\Omega_{ij} \in L^{(n,1)}(B,\R^n)$. Then $u$ is (H\"older-)continuous in $B$.
\end{theorem}
Observe that $\|\nabla Q\|_{L^n(B)} < \infty$ implies in particular that $Q \in VMO$. So in the spirit of linear theory of elliptic equations with VMO-coefficients, see e.g. \cite{IS98}, or scalar $p$-Laplace type equations, see e.g. \cite{BR13}, a result like \Cref{th:regularityLn1} might be expected. But again, our system is not in the standard form of $p$-Lapace systems because of $Q$ being a matrix.

For $Q \equiv I$ and in the scalar case $N=1$, \Cref{th:regularityLn1} was proven by Duzaar--Mingione \cite{DM10}. For $N \geq 2$, but still for $Q \equiv I$, while not being explicitly stated, it follows relatively easily with Kuusi--Mingione vectorial approximation techniques \cite{KM18}. Indeed, \cite{KM18} was a substantial inspiration behind the present work.
The main obstacle is that we have the rotation $Q \in \dot{W}^{1,n}(\R^n,SO(N))$ -- which does not fit in the framework of \cite{KM18} and makes the $n$-Laplace system nonstandard.

For example we were not able to prove a corresponding $n$-harmonic approximation result, \cite[Section 4]{KM18}. Nevertheless, heavily inspired by the Kuusi--Mingione techniques we circumvent the $n$-harmonic approximation argument with the help of Iwaniec' stability result, \cite{IS94}, which leads to an estimate (in our particular situation) that \emph{rhymes} with corresponding Kuusi--Mingione's estimates in \cite[Section 5]{KM18}. From this, by a covering argument we obtain the following Calder\'on-Zygmund type result, which is probably interesting on its own. See \Cref{Lninfty_estimate} for a more flexible statement we are going to use in the proofs.

\begin{proposition}\label{Lninfty_estimateintro}
	There exists a small $\tau > 0$, $\Gamma > 0$ and $\alpha\in (0,1)$
	such that for suitably small $\eps$ and even smaller $\gamma=\gamma(\eps)$ we have the following.
	
	Assume $f\in L^1(B^n(0,1);\R^N)$, $G\in L^\frac{n}{n-1}(B(0,1);\R^n\otimes \R^N)$, $Q\in W^{1,n}(B(0,1);SO(N))$ and $u \in W^{1,n}(B(0,1);\R^N)$ satisfy the system
	\begin{align*}
		-\di(|\g u|^{n-2} Q\g u) = f + \di G\ \text{in }B(0,1)
	\end{align*}
	with the bound
	\[
	\|\g Q\|_{L^n(B(0,1))} \leq \gamma.
	\]
	Then
	\[
	\| \g u \|_{L^{(n,\infty)}(B(0,\tau))} \leq \Gamma \left( \|f\|_{L^1(B(0,1))}^\frac{1}{n-1} + \|G\|_{L^\frac{n}{n-1}(B(0,1))}^\frac{1}{n-1} \right) + \Gamma \|\g u \|_{L^{n-\ve}(B(0,1))}+ \frac{1}{2} \| \g u \|_{L^{(n,\infty)}(B(0,1))}.
	\]
\end{proposition}

An adaptation of the argument in the proof of \Cref{th:regularityLn1} leads to the following corollary for the $n$-dimensional version of the prescribed mean curvature equation, shortly $H$-system. This has been considered by several authors also since the 1990s, e.g. \cite{DF91,DG00,MY96,MY96b,W99,Strzelecki03,K10,S13,SS17,FKLV18,MPS22} -- observe that \Cref{th:main1} is not directly applicable, because the the curl condition is likely not reasonable.
\begin{corollary}\label{co:Hsystem}
	Let $H: \R^{n+1} \to \R$ be bounded and globally Lipschitz. Assume $u\in W^{1,n}(B,\R^{n+1})$ satisfies
	\begin{align*}
		-\div(|\nabla u|^{n-2} \nabla u) &= H(u)\, \dr_1 u\times \dr_2 u\times \cdots \times \dr_n u \quad \text{in $B$}
	\end{align*}
	If $\g u \in L^{(n,\frac{n}{n-1})}(B,\R^n)$ then $u$ is continuous in $B$.
\end{corollary}

Before we get to the outline of this paper, let us discuss a few extensions, and limitations of our result. First of all, we assume $Q(x) \in SO(N)$ in \Cref{th:regularityLn1} for the sake of presentation and simplicity, because this is what we have from \Cref{th:main1}. It is elementary to adapt the result to $Q(x) \in GL(N)$, where $\|Q\|_{L^\infty} + \|Q^{-1} \|_{L^\infty} < \infty$. Additionally, we can replace the right-hand side in \eqref{eq:divQnpde} by $\Omega_{ij} \cdot |\nabla u|^{n-2} \nabla u^j  + f_i$ where $f_i \in L^p$ for $p>1$ and still have the same result. We leave the details to the reader.

\subsection*{Outline of the paper:} In \Cref{s:commie} we discuss preliminaries, such as norms and operators, Lorentz spaces, but also the commutator estimates we need. In \Cref{s:hodge} we also discuss a somewhat special case of a Hodge decomposition with estimates we couldn't readily find in the literature. In \Cref{s:gauge} we introduce Uhlenbeck's Coulomb gauge, i.e. the rotation $Q$ and show how after rotation the antisymmetric potential improves on the Lorentz space scale. This will show that \Cref{th:main1} is a consequence of \Cref{th:regularityLn1}. In \Cref{s:iwaniec} we begin by using Iwaniec' stability arguments to obtain $L^{n-\eps}$-estimates of rotated $n$-harmonic systems. While \Cref{la:easyiwaniec} is standard, \Cref{lm:distance_harm} is the estimate that will lead to estimates that we can use for a Kuusi--Mingione scheme as in \cite{KM18}. The latter we carry out in \Cref{s:lorentz}, where we prove a more flexible version of \Cref{Lninfty_estimateintro}, namely \Cref{Lninfty_estimate}. The latter we use in \Cref{s:iteration}  to obtain a decay estimate in the $L^{n-\eps}$ and $L^{(n,\infty)}$ norm from which we conclude H\"older continuity, i.e. we prove \Cref{th:regularityLn1}. In \Cref{s:applications} we apply our results to $n$-harmonic maps and establish \Cref{co:harmonicmap}, and in \Cref{s:proofHsystem} we prove the result for $H$-systems, \Cref{co:Hsystem}.

\subsection*{Notation:} Throughout the paper we will assume $n \geq 3$.
Given two numbers $A$ and $B$ and a parameter $p$, we denote $A\aleq_p B$ for an inequality of the form $A\leq cB$, where $c$ is a constant depending only on $n$, $N$ and $p$. We denote $A\aleq B$ when the constant $c$ in the inequality $A\leq cB$ depends only on $n$ and $N$.

\subsection*{Acknowledgement}
A substantial part of this work was carried out while D.M. and A.S. were visiting University of Bielefeld. We like to express our gratitude to the University for its hospitality. D.M.'s visit was partially funded through SFB 1283 and ANR BLADE-JC ANR-18-CE40-002. D.M. thanks Paul Laurain for initiating him to the subject and for his constant support and advice. A.S. is an Alexander-von-Humboldt Fellow. A.S. is funded by NSF Career DMS-2044898.

\section{Preliminaries and commutator estimates}\label{s:commie}
\subsection{Operators and norms}
We recall the definition of the Riesz transforms $\Rz f$ applied to $f \in L^p(\R^n)$, $p \in (1,\infty)$.
\[
\Rz f := \mathcal{F}^{-1} \left( \mathbf{i} \frac{\xi}{|\xi|} \mathcal{F} f(\xi) \right),
\]
where $\mathcal{F}$ is the Fourier transform. This can be equivalently written as
\[
\Rz f(x) = c P.V. \int_{\R^n} \frac{x-y}{|x-y|^{n+1}}\, f(y)\, dy.
\]
Observe that $\Rz f$ is a vector, $\Rz f = (\Rz_1 f,\ldots,\Rz_n f)$.

We will also use the fractional Laplacian
\[
\laps{s} f = \mathcal{F}^{-1} \Big( |\xi|^s \mathcal{F} f(\xi) \Big).
\]
It is useful to observe that Riesz transform and half-Laplacian combine to a derivative (and in this sense $\Rz_\alpha$ is a zero-order derivative in direction $\alpha$),
\[
\Rz_\alpha \laps{1} f = -c \partial_{\alpha} f.
\]
The BMO-norm for $b \in L^1_{loc}(\R^n)$ is given by
\[
[b]_{BMO} := \sup_{x \in \R^n, r > 0} \mvint_{B(x,r)} |b-(b)_{B(x,r)}|,
\]
where we denote the mean value
\[
(b)_{B(x,r)} := \frac{1}{|B(x,r)|} \int_{B(x,r)} b = \mvint_{B(x,r)} b.
\]

\subsection{Lorentz spaces and maximal functions}\label{s:lorentzspace}
Here we recall the definitions and relevant properties of Lorentz spaces, with applications to maximal functions. For further reading, see \cite{bennett1988,grafakos2014}.

Let $\Omega\subset \R^n$ be an open set. Given a function $f : \Omega \to \R$, we define its decreasing rearrangement $f^* : [0,\infty)\to [0,\infty)$ by
\begin{align*}
	\forall t>0, \ \ \ f^*(t) := \inf\left\{ \lambda\geq 0 : |\{ x\in \Omega : |f(x)|>\lambda\}| \leq s \right\}.
\end{align*}
We also consider the following transformation of $f^*$ :
\begin{align*}
	\forall t>0, \ \ \ f^{**}(s) := \frac{1}{s} \int_0^s f^*(t) dt.
\end{align*}
Given $p\in(0,\infty)$ and $q\in(0,+\infty]$, we define the Lorentz norm $L^{(p,q)}(\Omega)$ as
\begin{align*}
	\| f\|_{L^{(p,q)}(\Omega)} := \left\{ \begin{array}{l l}
		\left(  \int_0^\infty \left( f^{**}(s) s^\frac{1}{p} \right)^q \frac{ds}{s} \right)^\frac{1}{q} & \text{ if }q<\infty,\\
		\sup_{s>0} s^\frac{1}{p} f^{**}(s) & \text{ if }q=\infty.
	\end{array}
	\right.
\end{align*}
This quantity is actually a norm only when $p,q\geq 1$. We also define equivalent quantities which are more useful in practice:
\begin{align*}
	[f]_{L^{(p,q)}(\Omega)} :=\left\{ \begin{array}{l l}
		\left(  \int_0^\infty \left( f^*(s) s^\frac{1}{p} \right)^q \frac{ds}{s} \right)^\frac{1}{q} & \text{ if }q<\infty,\\
		\sup_{s>0} s^\frac{1}{p} f^*(s) & \text{ if }q=\infty.
	\end{array}
	\right.
\end{align*}
When $q=\infty$, we will use an alternative form
\begin{align*}
	[f]_{L^{(p,\infty)}(\Omega)} &= \sup_{\lambda>0} \lambda |\{ x\in \Omega : |f(x)| > \lambda \}|^\frac{1}{p}.
\end{align*}
For any $p\in(0,\infty)$ and $q\in(0,+\infty]$, there exists a universal constant $c(p,q)>1$ such that for any $f\in L^{(p,q)}(\Omega)$
\begin{align*}
	[f]_{L^{(p,q)}(\Omega)} &\leq \|f\|_{L^{(p,q)}(\Omega)} \leq c(p,q) [f]_{L^{(p,q)}(\Omega)}.
\end{align*}
Lorentz spaces are refinements of the Lebesgue spaces in the following sense. Given $1\leq q<p<r \leq +\infty$ and $|\Omega|<\infty$, it holds
\begin{align*}
	L^r(\Omega) = L^{(r,r)}(\Omega) \subset L^{(p,r)}(\Omega) \subset L^{(p,p)}(\Omega) = L^p(\Omega) \subset L^{(p,q)}(\Omega) \subset L^{(q,q)}(\Omega) = L^q(\Omega).
\end{align*}
The Lorentz norms have the same behaviour as the Lebesgue spaces in the following sense :
\begin{itemize}
	\item There is a Hölder inequality for Lorentz spaces. Given $p_1,p_2,q_1,q_2\in[1,\infty]$ such that $1=\frac{1}{p_1} + \frac{1}{p_2}$ and $1=\frac{1}{q_1} + \frac{1}{q_2}$, there exists $c=c(p_1,p_2,q_1,q_2)>0$ such that for any $f\in L^{(p_1,q_1)}(\Omega)$ and $g\in L^{(p_2,q_2)}(\Omega)$ :
	\begin{align*}
		\|fg\|_{L^1(\Omega)} &\leq c\|f\|_{L^{(p_1,q_1)}(\Omega)} \|g\|_{L^{(p_2,q_2)}(\Omega)}.
	\end{align*}
	\item Given $r>0$ and $p\in(0,\infty)$ and $q\in (0,\infty]$, it holds
	\begin{align*}
		\|f^r \|_{L^{(p,q)}(\Omega)} &= \|f\|_{L^{(pr,qr)}(\Omega)}^r.
	\end{align*}
\end{itemize}
Given a cube $Q_0 \subset \R^n$, we define the sharp maximal function : for any $x\in Q_0$,
\begin{align*}
	M_{Q_0}^\sharp f (x) := \sup\left\{ \mvint_{Q} |f-(f)_Q| : Q\subset Q_0, x\in Q,\ Q\text{ is a cube} \right\}.
\end{align*}
Thanks to \cite[Chapter 5, Theorem 7.3, p.377]{bennett1988}, we have a pointwise lower bound on the rearrangements $M^\sharp_{Q_0}$:
\begin{theorem}\label{th:Msharp_star}
	There exists a universal constant $c=c(n)>0$ such that for any cube $Q_0\subset \R^n$ and any $f\in L^1(Q_0)$, it holds
	\begin{align*}
		\forall t\in\left( 0, \frac{|Q_0|}{6} \right), \ \ \ f^{**}(t) - f^*(t) \leq c(M^\sharp_{Q_0} f)^*(t).
	\end{align*}
\end{theorem}

\subsection{Commutator estimates}

It is well known that the Riesz transforms are bounded on $L^p(\R^n)$, namely
\begin{theorem}[Calder\'on-Zygmund theorem]\label{th:CZ}
	For $p \in (1,\infty)$ there exists a constant $C$ such that the following holds for all $f \in L^p(\R^n)$
	\[
	\|\Rz f\|_{L^p(\R^n)} \leq C \|f\|_{L^p(\R^n)}.
	\]
\end{theorem}

When $f$ is a product of functions, cancellations can appear, which is the harmonic analysis reason for many compensated compactness type results. The fundamental estimate is the following

\begin{theorem}[Coifman-Rochberg-Weiss commutator theorem, \cite{CRW76}]\label{th:crw}
	For $p \in (1,\infty)$ there exists a constant $C$ such that the following holds.
	
	For any $b \in L^\infty(\R^n)$ and $f \in L^p(\R^n)$ we have
	\[
	[\Rz,b](f) := \Rz (bf) - b\Rz (f) \in L^p(\R^n)
	\]
	with the estimate
	\[
	\|[\Rz,b](f)\|_{L^p(\R^n)} \leq C\, [b]_{BMO}\, \|f\|_{L^p(\R^n)}.
	\]
\end{theorem}

We need a special version in Lorentz spaces, that can be proven e.g. with the methods in \cite[Theorem 6.1.]{LS20}.
\begin{theorem}[Coifman-Rochberg-Weiss commutator theorem, \cite{CRW76}]\label{th:crwlorentz}
	For $p \in (1,\infty)$, $q_1,q_2,q_3 \in [1,\infty]$ such that
	\[
	\frac{1}{q_1} = \frac{1}{q_2} + \frac{1}{q_3}
	\] there exists a constant $C$ such that the following holds.
	
	For any $b \in L^\infty(\R^n)$ and $f \in L^p(\R^n)$ we have the estimate
	\[
	\|[\Rz,b](f)\|_{L^{(p,q_1)}(\R^n)} \leq C\, \|\nabla b\|_{L^{(n,q_2)}(\R^n)}\, \|f\|_{L^{(p,q_3)}(\R^n)}.
	\]
\end{theorem}

For vectorfields $G \in L^p(\R^n,\R^n)$ we will use the following notation, the zero-order divergence $\Rz \cdot$,
\[
\|\Rz \cdot G\|_{L^p(\R^n)} := \|\sum_{\alpha=1}^n\Rz_\alpha G_\alpha\|_{L^p(\R^n)},
\]
and the zero-order curl, $\Rz^\perp$,
\[
\left\| \Rz^\perp G \right\|_{L^p(\R^n)} := \max_{\alpha, \beta \in \{1,\ldots,n\}} \|\Rz_\alpha G_\beta-\Rz_\beta G_\alpha\|_{L^p(\R^n)},
\]
We have the following stability result by Iwaniec--Sbordone \cite{IS94}. For a proof, inspired by \cite{RW83}, we refer to \cite[Theorem 13.2.1]{IM01}. Observe that $\Rz^\perp \nabla u = 0$.

\begin{theorem}[Iwaniec stability result \cite{IS94}]\label{th:iwaniec}
	For any $1 < p_1 < p_2 < \infty$ there exists $\eps_0 > 0$ and a constant $C=C(p_1,p_2)$ such that whenever $\eps \in (-\eps_0,\eps_0)$ and $p \in (p_1,p_2)$ then
	\[
	\left\|\Rz^\perp \brac{|\nabla u|^\eps \nabla u} \right\|_{L^{\frac{p}{1+\eps}}(\R^n)} \leq C\, |\eps| \|\nabla u\|_{L^p(\R^n)}^{1+\eps}
	\]
	whenever the right-hand side is finite.
\end{theorem}
Let us remark that it is unclear to what extend the statement of \Cref{th:iwaniec} holds for Lorentz spaces (with uniform constant $C$). The operator in question is nonlinear, and not always Lipschitz, so interpolation obtains unsatisfactory dependencies of $C$ on $\eps$; See also \cite[Problem 3.2.]{SS17}.

\subsection{Hodge decomposition}\label{s:hodge}

It is not difficult to show, \cite[(10.68),(10.69),(10.70)]{IM01}, that for any $f \in L^p(\R^n;\R^n)$, $p \in (1,\infty)$, $q \in [1,\infty]$, we have
\begin{equation}\label{eq:Rnhodge}
	\|f\|_{L^{(p,q)}(\R^n)} \aeq \|\Rz \cdot f\|_{L^{(p,q)}(\R^n)} + \|\Rz^\perp f\|_{L^{(p,q)}(\R^n)}
\end{equation}
This is essentially a consequence of Hodge decomposition in $\R^n$.

The main goal of this section is the following Hodge decomposition on a ball without harmonic term.
%
\begin{proposition}[Hodge decomposition]\label{pr:hodge}
	Let $p \geq 2$. Assume that $F \in C_c^\infty(B,\R^n)$ then there exist $a \in W^{1,p}_0(B(0,1))$ and $B \in W^{1,p}(B(0,1),\R^n)$ such that
	\begin{equation}\label{eq:mainhodge}
		F = \nabla a + B\quad \text{in $B(0,1)$}
	\end{equation}
	and we have the estimates
	\begin{equation}\label{eq:mainhodgeesta}
		\|a\|_{W^{1,p}(B(0,1))} \aleq \|\Rz \cdot F\|_{L^p(\R^n)},
	\end{equation}
	and
	\begin{equation}\label{eq:mainhodgeestB}
		\|B\|_{L^{p}(B(0,1))} \aleq \|\Rz^\perp F\|_{L^p(\R^n)}
	\end{equation}
\end{proposition}

Decompositions such as the one in \eqref{eq:mainhodge} are well-known, cf. \cite{ISS99,CDK12}. What is, however, crucial for us are the estimates \eqref{eq:mainhodgeesta}, \eqref{eq:mainhodgeestB} -- and we are not aware of them being readily available in the literature.

In order to prove \Cref{pr:hodge} we will use the notion of differential forms, for which we refer the reader to \cite{ISS99,CDK12}. With
\[
i: \partial B(0,1) \to \overline{B(0,1)}
\]
we denote the inclusion map. We write $\Ep^k B(0,1)$ the set of functions defined on $B(0,1)$ with values into $k$-forms on $\R^n$. Given $\omega \in \Ep^k B(0,1)$, we can write $\omega$ in coordinates
\[
\omega = \sum_{I \in \mathcal{I}} \omega_{I} dx_I,
\]
where $\mathcal{I}= \{ (i_1,\ldots,i_k), \quad 1\leq i_1 < i_2 < \ldots < i_k \leq n\}$ are ordered tuples and $\omega_I$ are functions.
We say that $\omega$ belongs to $W^{1,p}(\Ep^{k} B(0,1))$ if $\omega_I \in W^{1,p}(B(0,1))$ for all $I \in \mathcal{I}$.

By $d$ we denote the differential and by $d^\hdg= \hdg d \hdg$ we denote the co-differential, where $\hdg$ is the Hodge star operator.

\begin{lemma}\label{la:hodge:aguy}
	Assume $p \in [2,\infty)$. Assume $F \in C_c^\infty(\Ep^1 B(0,1))$. Then there exists
	$a \in W^{1,p}_0(B(0,1))$ such that
	\begin{equation}\label{eq:aw1phodge}
		\begin{cases}
			d^\hdg da = d^\hdg F \quad  &\text{in $B(0,1)$,}\\
			a = 0 \quad &\text{on $\partial B(0,1)$,}
		\end{cases}
	\end{equation}
	and we have
	\begin{equation}\label{eq:aw1phodgeest}
		\|a\|_{W^{1,p}(B(0,1))} \aleq \|\Rz \cdot F\|_{L^p(\R^n)}.
	\end{equation}
\end{lemma}
\begin{proof}
	We simply solve
	\[
	\begin{cases}
		\lap a = \div F \quad  &\text{in $B(0,1)$}\\
		a = 0 \quad &\text{on $\partial B(0,1)$}.
	\end{cases}
	\]
	Using that $a$ is a zero-form, and thus $d^\hdg a = 0$ we observe that this is equivalent to \eqref{eq:aw1phodge}.
	
	For the estimate \eqref{eq:aw1phodgeest} we let $\varphi \in C_c^\infty(B(0,1))$. By the support of $F$
	\[
	\int_{B(0,1)} \div F \varphi = \int_{\R^n} \div F\, \varphi = \int_{\R^n} \Rz \cdot F\, \laps{1} \varphi \aleq \|\Rz \cdot F\|_{L^p(\R^n)}\, \|\laps{1} \varphi\|_{L^{p'}(\R^n)}
	\]
	Since
	\[
	\|\laps{1} \varphi\|_{L^{p'}(\R^n)} \aleq \|\nabla \varphi\|_{L^{p'}(\R^n)} \aleq \|\varphi\|_{W^{1,p'}(B(0,1))},
	\]
	we conclude that
	\[
	\|\lap a\|_{\brac{W^{1,p'}_0(B(0,1))}^\ast} \aleq \|\Rz \cdot F\|_{L^p(\R^n)}.
	\]
	Standard Calder\'on-Zygmund estimates imply \eqref{eq:aw1phodgeest}.
\end{proof}

\begin{lemma}\label{la:hodge:bguy}
	Let $p \in [2,\infty)$. Assume $F \in C_c^\infty(\Ep^1 B(0,1))$. Then there exists $b \in L^p(\Ep^2 B(0,1)) \cap C^1(\Ep^2 B(0,1))$, such that
	\begin{equation}\label{eq:hdgbguy}
		\begin{cases}
			d d^\hdg b= dF&\text{in $B(0,1)$},\\
			db = 0 \quad &\text{in $B(0,1)$},\\
			i^\ast(d^\hdg b) = 0 \quad &\text{on $\partial B(0,1)$},
		\end{cases}
	\end{equation}
	such that
	\begin{equation}\label{eq:hdgbguyest}
		\|b\|_{W^{1,p}(B(0,1))} \aleq \|\Rz^\perp F\|_{L^p(\R^n)}.
	\end{equation}
\end{lemma}
\begin{proof}
	Let
	\[
	X = \left \{\eta \in W^{1,2}\brac{\Ep^2 B(0,1)},\, d\eta = 0\ \text{in $B(0,1)$}\right \},
	\]
	which is clearly a closed subset of $W^{1,2}\brac{\Ep^2 B(0,1)}$.
	Consider the energy
	\[
	\mathcal{E}(b) := \int_{B(0,1)} \frac{1}{2}\langle d^\hdg b, d^\hdg b\rangle + \langle F, d^\hdg b\rangle.
	\]
	Let $(b_k)_{k \in \N} \subset X$ be a minimizing sequence in $X$ for $\mathcal{E}$.
	
	Clearly we may assume $\mathcal{E}(b_k) \leq \mathcal{E}(0) = 0$, and then we have
	\[
	\sup_{k} \|d^\hdg b_k\|_{L^2(B(0,1))}^2 \aleq \|F\|_{L^2(B(0,1))}^2.
	\]
	By \cite[Theorem 6.3.]{ISS99}, there exists $h_k \in W^{1,2}(\Ep^2 B(0,1))$, $dh_k =0$, $d^\hdg h_k =0$, \cite[(3.18)]{ISS99}, such that
	\[
	\|b_k - h_k\|_{W^{1,2}(\Ep^2 B(0,1))} \aleq \underbrace{\|db_k\|_{L^2(B(0,1))}}_{=0} + \|d^\hdg b_k\|_{L^2(B(0,1))},
	\]
	i.e.
	\[
	\|b_k - h_k\|_{W^{1,2}(\Ep^2 B(0,1))} \aleq \|d^\hdg (b_k-h_k)\|_{L^2(B(0,1))} \aleq \|F\|_{L^2(B(0,1))}.
	\]
	Observe that $b_k-h_k$ still belongs to $X$, so we might as well assume $h_k = 0$.
	
	Then $b_k$ is uniformly bounded in $W^{1,2}(\Ep^2 B(0,1))$, and up to taking a subsequence converges weakly in $W^{1,2}(\Ep^2 B(0,1))$ to some $b \in X$. Clearly we have weak lower semicontinuity of $\mathcal{E}(b)$ w.r.t. weak $W^{1,2}$-convergence, so we have found a minimizer $b \in X$ of $\mathcal{E}$.
	
	Now for the equation: For $\phi \in C^\infty(\Ep^2 \overline{B})$ with $d\phi = 0$ then $b+t\phi \in X$, and $\frac{d}{dt} \Big|_{t=0}\mathcal{E}(b+t\phi) = 0$ implies
	\[
	\int_{B} \langle d^\hdg b, d^\hdg \phi\rangle + \langle F, d^\hdg \phi\rangle = 0.
	\]
	For a general $\varphi \in C^\infty(\Ep^2 \overline{B(0,1)})$, cf. \cite[(5.8)]{ISS99}, we find $\phi \in C^\infty(\Ep^2 \overline{B(0,1)})$, $d\phi = 0$ in $B(0,1)$, and $\psi \in C^\infty(\Ep^3 \overline{B(0,1)})$ such that
	\[
	\varphi = \phi + d^\hdg \psi
	\]
	Since $d^\hdg \circ d^\hdg = 0$ we conclude
	\begin{equation}\label{eq:pdedhdg}
		\int_{B(0,1)} \langle d^\hdg b, d^\hdg \varphi\rangle + \langle F, d^\hdg \varphi\rangle = 0 \quad \text{for all $\varphi\in C^\infty(\Ep^2 \overline{B(0,1)})$}.
	\end{equation}
	An integration by parts, \cite[Theorem 3.28]{CDK12}, implies
	\[
	\begin{split}
		&\int_{B(0,1)} \langle d^\hdg b+F, d^\hdg \varphi\rangle+\int_{B(0,1)} \langle dd^\hdg b+dF, \varphi\rangle\\
		=&\int_{\partial B(0,1)} \langle \nu \wedge \brac{d^\hdg b+F}, \varphi\rangle,
	\end{split}
	\]
	where $\nu\in \Ep^1 \partial B(0,1)$ denotes the outwards facing unit norm (as a form).
	This gives the interior equation \eqref{eq:hdgbguy} for $b$. For the boundary data we observe from the above equality that
	\[
	\nu \wedge (d^\hdg b+F)  = 0 \quad \text{on $\partial B(0,1)$}.
	\]
	By \cite[Corollary 3.21]{CDK12}, this is equivalent to the condition
	\[
	i^* (d^\hdg b+F)  = 0 \quad \text{on $\partial B(0,1)$}.
	\]
	Since $F$ has zero boundary data we have $i^\ast(F) = 0$, so $i^\ast(d^\hdg b) = 0$. Thus $b$ is a solution to the equation \eqref{eq:hdgbguy}.
	In the previous arguments we have already used that $b$ is smooth, which follows as in \cite[Theorem 6.12]{CDK12}, but
	our main burden is to prove the estimate \eqref{eq:hdgbguyest}.
	
	Let $\psi \in C^\infty(\Ep^2\overline{B(0,1)})$, $\psi = \psi_{\alpha \beta} dx^\alpha \wedge dx^\beta$, and denote by $E\psi \in C_c^\infty(\Ep^2\R^n)$ an extension with
	\[
	\|E\psi\|_{W^{k,q}(\R^n)} \aleq_{k,q} \|\psi\|_{W^{k,q}(B(0,1))}.
	\]
	Since $F=F_\alpha dx^\alpha$ vanishes outside of $B(0,1)$, we have
	\[
	\begin{split}
		\int_{B(0,1)} \langle F, d^\hdg \psi\rangle =& \int_{\R^n} \langle F, d^\hdg E\psi\rangle\\
		=&  \sum_{\alpha,\beta}\int_{\R^n} F_\beta \, \partial_\alpha E\psi_{\alpha \beta}\\
		\overset{\psi_{\alpha \beta} = -\psi_{\beta \alpha}}{=}&  \sum_{\alpha,\beta}\int_{\R^n} \brac{F_\beta \, \partial_\alpha E\psi_{\alpha \beta}-F_\alpha \, \partial_\beta E\psi_{\alpha \beta}}\\
		=&\sum_{\alpha,\beta}\int_{\R^n} \brac{\Rz_\alpha \brac{F_\beta } -\Rz_\beta\brac{F_\alpha}} \laps{1} E\psi_{\alpha \beta}\\
		\aleq&\|\Rz^\perp F\|_{L^p(\R^n)} \|\psi\|_{W^{1,p'}(B(0,1))}\\
	\end{split}
	\]
	From \eqref{eq:pdedhdg} we thus conclude
	\begin{equation}\label{eq:dastbest}
		\abs{\int_{B(0,1)} \langle d^\hdg b, d^\hdg \psi\rangle } \aleq \|\Rz^\perp F\|_{L^p(\R^n)} \|\psi\|_{W^{1,p'}(\Ep^1 B(0,1))}.
	\end{equation}
	By duality, we may find $\eta \in C_c^\infty(\Ep^1 B)$ with $\|\eta\|_{L^{p'}(B(0,1))} \leq 1$ such that
	\[
	\|d^\hdg b\|_{L^{p}(B(0,1))} \leq 2\abs{\int_{B(0,1)} \langle d^\hdg b, \eta\rangle }.
	\]
	By Hodge decomposition of $\eta$, \cite[(5.16)]{ISS99},
	\[
	\eta = d^\hdg \psi + \phi \quad \text{in }B(0,1),
	\]
	where $\phi = d\alpha + h$, with $\alpha \in W^{1,p}(\Ep^0 B(0,1))$ satisfying $i^*(d\alpha)=0$ on $\dr B(0,1)$ and $h \in L^p(\Ep^1 B(0,1))$ verifying $dh=0$, $d^\hdg h =0$ in $B(0,1)$ and $i^*h=0$ on $\dr B$. Thus, \cite[Corollary 3.21.]{CDK12} and \cite[Theorem 3.23.]{CDK12}
	\[
	\begin{cases}
		d \phi = 0 \quad &\text{in $B(0,1)$}\\
		\nu \wedge (d\alpha + h) = 0 \quad&\text{on $\partial B(0,1)$}.
	\end{cases}
	\]
	and we have \cite[(5.18)]{ISS99},
	\[
	\|\psi\|_{W^{1,p}(B(0,1))} \aleq \|\eta\|_{L^p(B(0,1))}.
	\]
	Integration by parts, \cite[Theorem 3.28]{CDK12}, implies
	\[
	\int_{B(0,1)} \langle d^\hdg b, \phi\rangle + \int_{B(0,1)} \langle \underbrace{d\phi}_{=0},b\rangle = \int_{\partial B(0,1)} \langle \underbrace{\nu \wedge \phi}_{=0},b\rangle.
	\]
	Consequently,
	\[
	\begin{split}
		\|d^\hdg b\|_{L^{p}(B(0,1))} \leq &2\abs{\int_{B(0,1)} \langle d^\hdg b, d^\hdg \psi\rangle }\\
		\overset{\eqref{eq:dastbest}}{\aleq}&\|\Rz^\perp F\|_{L^p(\R^n)}.
	\end{split}
	\]
	We apply once more \cite[Theorem 6.3.]{ISS99} to find $h \in W^{1,p}(\Ep^2 B(0,1))$, $dh  =0$, $d^\hdg h = 0$, such that
	\[
	\|b-h\|_{W^{1,p}(\Ep^2 B(0,1))} \aleq \|\underbrace{db}_{=0}\|_{L^p(\Ep^3 B(0,1))} \|d^\hdg b\|_{L^{p}(\Ep^2 B(0,1))} \aleq \|\Rz^\perp F\|_{L^p(\R^n)}.
	\]
	Setting $\bar{b} := b-h$ we observe that $\bar{b}$ still satisfies \eqref{eq:hdgbguy} but also \eqref{eq:hdgbguyest}.
	We can conclude.
\end{proof}

\begin{proof}[Proof of \Cref{pr:hodge}]
	We take $a$ from \Cref{la:hodge:aguy} and $b$ from \Cref{la:hodge:bguy}, which satisfy the claimed estimates.
	
	Set
	\[
	h := F - da -d^\hdg b \in C^\infty(\Ep^1 \overline{B(0,1)}).
	\]
	We can conclude once we show
	\begin{equation}\label{eq:hequiv0hodge}h \equiv 0 \quad \text{in $B(0,1)$}.\end{equation}
	For \eqref{eq:hequiv0hodge} we observe by the choice of $b$ from \Cref{la:hodge:bguy},
	\[
	dh = dF - dd^\hdg b = 0 \quad \text{in $B(0,1)$}
	\]
	and by the choice of $a$ from \Cref{la:hodge:aguy}
	\[
	d^\hdg h = d^\hdg F - dd^\hdg a \quad \text{in $B(0,1)$}.
	\]
	So $h$ is harmonic, $dh = 0$, $d^\hdg h = 0$.
	
	By assumption we have
	\[
	i^\ast F = 0 \quad \text{on $\partial B(0,1)$}
	\]
	Since
	\[
	i^\ast(a) = 0, \quad \text{so } i^\ast(da) = 0,
	\]
	and
	\[
	i^\ast (d^\hdg b) = 0,
	\]
	we conclude $i^\ast h = 0$. By Poincar\'e lemma we may write $h = d\tilde{h}$, where $\tilde{h}$ is a $0$-form, that is to say a function. Then we have
	
	\[
	\begin{cases}
		\lap \tilde{h} = 0 \quad &\text{in $B(0,1)$}\\
		\partial_{\tau} h = 0 \quad &\text{on $\partial B(0,1)$}
	\end{cases}
	\]
	The boundary condition means that $\tilde{h}$ is constant on $\partial B(0,1)$, thus $\tilde{h}$ is constant, and consequently $h = d\tilde{h} = 0$. We have established \eqref{eq:hequiv0hodge} and we can conclude.
\end{proof}

Our goal in \Cref{s:iwaniec} is to apply \Cref{pr:hodge} to terms of the form
\[
F_i = Q_{ij} |\nabla w|^{-\eps} \nabla w^j
\]
where $w \in W^{1,n-\eps}_0(B(0,1))$ and $Q_{ij} \in L^\infty \cap W^{1,n}(B(0,1))$.

Observe that this implies that for $p\in [2,\frac{n-\eps}{1-\eps}]$ we can approximate $F_i$ by $F_{i;\delta} \in L^p(\Ep^1\R^n) \cap C_c^\infty(\Ep^1 B(0,1))$, by considering
\[
F_{i;\delta} = Q_{ij;\delta} (|\nabla w_\delta|+\delta)^{-\eps} \nabla w^j_{\delta}
\]
and $w^j_\delta \in C_c^\infty(B(0,1))$, $Q_{ij;\delta} \in C^\infty(\R^n,\R^{N \times N})$ the usual approximations, and we have
\[
\|F_{i;\delta}-F_{i}\|_{L^p(\R^n)} \xrightarrow{\delta \to 0} 0.
\]
That is, we have the following as a consequence of \Cref{pr:hodge}.
\begin{corollary}\label{co:ourhodge}
	Assume $w \in W^{1,n}_0(B(0,1))$ and $Q \in W^{1,n}(B(0,1),\R^{n \times n})$.
	There exist $a^i \in W^{1,p}_0(B(0,1))$ and $B^i \in W^{1,p}(B(0,1),\R^n)$ such that
	\[
	Q_{ij} |\nabla w|^{-\eps} \nabla w^j = \nabla a^i + B^i\quad \text{in $B(0,1)$}
	\]
	and we have the following estimates for all small $\eps > 0$
	\[
	\|a\|_{W^{1,\frac{n-\eps}{1-\eps}}(B(0,1))} \aleq \|\tilde{Q}\|_{L^\infty} \|\nabla w\|_{L^{n-\eps}(B(0,1))}^{1-\eps},
	\]
	and
	\[
	\|B\|_{L^{\frac{n-\eps}{1-\eps}}(B(0,1))} \aleq \brac{|\eps|\, \|\tilde{Q}\|_{L^\infty}+[\tilde{Q}]_{BMO}} \|\nabla w\|_{L^{n-\eps}(B(0,1))}^{1-\eps}
	\]
	Here $\tilde{Q}$ is any extension of $Q$ to $\R^n$.
\end{corollary}
\begin{proof}
	The estimate for $a$ follows from boundedness of the  Riesz transform, \Cref{th:CZ}, and H\"older's inequality.
	The estimate for $B$ is a consequence of commutator estimates,
	\[
	\begin{split}
		&\|\Rz^\perp \brac{\tilde{Q}_{ij} |\nabla w|^{-\eps} \nabla w^j}\|_{L^{\frac{n-\eps}{1-\eps}}(\R^n)}\\
		\aleq &[\tilde{Q}]_{BMO} \|\nabla w\|_{L^{n-\eps}(B(0,1))}^{1-\eps} + \|\tilde{Q}\|_{L^\infty} \|\Rz^\perp \brac{|\nabla w|^{-\eps} \nabla w^j}\|_{L^{\frac{n-\eps}{1-\eps}}(\R^n)}\\
		\aleq &[\tilde{Q}]_{BMO} \|\nabla w\|_{L^{n-\eps}(B(0,1))}^{1-\eps} + \|\tilde{Q}\|_{L^\infty} |\eps| \|\nabla w\|_{L^{n-\eps}(B(0,1))}^{1-\eps}\\
	\end{split}
	\]
	where in the first line we used \Cref{th:crw} and in the second \Cref{th:iwaniec}.
\end{proof}

\section{Choice of gauge: Theorem~\ref{th:regularityLn1} implies Theorem~\ref{th:main1}}\label{s:gauge}
The Uhlenbeck gauge was introduced in \cite{U82}, the application to harmonic maps is due to Rivi\`{e}re \cite{R07}. See also \cite{S10} for a variational method, \cite[Theorem 1.2]{goldstein2018} for $n \geq 3$. The following result is also often used in the theory of wave maps, see e.g. \cite{SS02}.

\begin{proposition}[Uhlenbeck's Coulomb gauge]\label{pr:uhlenbeck}
	There exists $\gamma > 0$ and $C > 0$ such that the following holds.
	
	Assume that $\Omega_{ij} \in L^n(\R^n,\R^n)$ is antisymmetric, i.e. $\Omega_{ij} = -\Omega_{ji}$, and assume
	\[
	\|\Omega_{ij}\|_{L^n(\R^n)} \leq \gamma
	\]
	Then there exists $Q \in L^\infty(\R^n,SO(N))$, such that
	\[
	\|\nabla Q\|_{L^n(\R^n)} \leq C\|\Omega\|_{L^n(\R^n)},
	\]
	and we have for
	\[
	\Omega^Q := Q \nabla Q^T + Q \Omega Q^T
	\]
	
	\begin{equation}\label{eq:uhlenbeckgauge}
		\div (\Omega^Q ) = 0 \quad \text{in $\R^n$}.
	\end{equation}
\end{proposition}
Since \Cref{pr:uhlenbeck} in this form is usually stated on domains such as balls, or with $W^{1,\frac{n}{2}}$-assumptions on $\Omega$, we sketch the proof.
\begin{proof}
	Fix $R > 0$ and set
	\[
	\Omega_R(x) := R \Omega(R\cdot).
	\]
	Observe that
	\[
	\|\Omega\|_{L^n(B(0,1))} = \|\Omega\|_{L^n(B(0,R))} \leq \gamma.
	\]
	By e.g. \cite[Theorem 1.2]{goldstein2018}, if $\gamma$ is small enough, there exists $\tilde{P}_R \in W^{1,n}(B(0,1),SO(N))$ such that
	\[
	\div(\tilde{P}_R \nabla \tilde{P}_R^T + \tilde{P}_R \Omega_R \tilde{P}_R^T) = 0 \quad \text{in $B(0,1)$}
	\]
	and $\tilde{P}_R - I \in W^{1,n}_0(B(0,1),\R^{N \times N})$, where $I$ is the identity matrix in $\R^{N \times n}$, and
	\[
	\|\nabla \tilde{P}_R\|_{L^n(B(0,1))} \aleq_{n,N} \|\Omega_R\|_{L^n(B(0,1))} \leq \|\Omega\|_{L^n(\R^n)}.
	\]
	Setting
	\[
	P_R := \begin{cases}
		\tilde{P}_R (\cdot / R) \quad& \text{in $B(0,R)$}\\
		I \quad &\text{in $\R^n \setminus B(0,R)$}
	\end{cases}
	\]
	we have
	\[
	\div(P_R \nabla P_R^T + P_R \Omega P_R^T) = 0 \quad \text{in $B(0,R)$},
	\]
	moreover $P_R \in W_{loc}^{1,n}(\R^n,SO(N))$, and
	\[
	\|\nabla P_R\|_{L^n(\R^n)} \aleq_{n,N} \|\Omega\|_{L^n(\R^n)}.
	\]
	In particular, up to taking a subsequence $R \to \infty$ we find that $P_R$ converges to some $Q \in L^\infty(\R^n,SO(N))$, $\nabla Q \in L^n(\R^n)$ and, using Rellich's theorem
	\[
	\div(Q \nabla Q^T + Q \Omega Q^T) =0 \quad \text{in $\R^n$}
	\]
	and
	\[
	\|\nabla Q\|_{L^n(\R^n)} \leq \liminf_{R \to \infty} \|\nabla P_R\|_{L^n(\R^n)} \aleq_{n,N} \|\Omega\|_{L^n(\R^n)}.
	\]
\end{proof}

Our main point of view is that instead of using the divergence-free condition \eqref{eq:uhlenbeckgauge} directly in the equation, as is usually done in the ($n$-)harmonic map theory, we use the condition to improve the \emph{regularity} of $\Omega^Q$, as is often done, e.g., in wave maps \cite{SS02}. Namely, the Uhlenbeck gauge regularizes with the regularity of $\Omega$ and $\curl \Omega$. Precisely we have
\begin{proposition}\label{pr:uhlenbeckgauge}
	Take $\gamma$ sufficiently small -- possibly smaller than the one \Cref{pr:uhlenbeck}. Then under the assumptions of \Cref{pr:uhlenbeck} we have additionally
	\[
	\|\nabla Q\|_{L^{(n,2)}(\R^n)} \aleq \|\Omega\|_{L^{(n,2)}(\R^n)},
	\]
	and
	\[
	\|\Omega^Q \|_{L^{(n,1)}(\R^n)} \aleq \|\Omega\|_{L^{(n,2)}(\R^n)}^2 + \|\Rz^\perp \Omega\|_{L^{(n,1)}(\R^n)},
	\]
	whenever the right-hand sides are finite.
\end{proposition}
\begin{proof}
	From \eqref{eq:uhlenbeckgauge} we have
	\[
	\Rz \cdot (Q \nabla Q^T + Q \Omega^T Q) =0 \quad \text{in $\R^n$}.
	\]
	Since $c\laps{1} = -\Rz \cdot \nabla$,
	\[
	c\laps{1} Q^T = Q^T [\Rz \cdot,Q](\nabla Q^T) + Q^T \Rz\cdot (Q \Omega Q^T) \quad \text{in $\R^n$}.
	\]
	Consequently we have in view of \Cref{th:crwlorentz} for any $q \in [1,\infty]$,
	\[
	\|\laps{1} Q\|_{L^{(n,q)}(\R^n)} \aleq \|Q\|_{L^\infty}\, \|\nabla Q\|_{L^{n,2q}(\R^n)}^2 + \|\Omega\|_{L^{n,q}(\R^n)}.
	\]
	Applying this estimate iteratively for $q =\frac{n}{2},\frac{n}{4},\ldots,2$, using that $L^{(n,n)} = L^n$, and observing that since $\gamma < 1$ we can estimate high powers of $\gamma$ by $\gamma$ for convenience,
	\[
	\|\nabla Q\|_{L^{(n,2)}(\R^n)} \aleq \|\nabla Q\|_{L^{n}(\R^n)} + \|\Omega\|_{L^{(n,2)}(\R^n)} \aleq \|\Omega\|_{L^{n}(\R^n)} + \|\Omega\|_{L^{(n,2)}(\R^n)} \aleq 2\|\Omega\|_{L^{(n,2)}(\R^n)} .
	\]
	Now we observe that $\Rz^\perp \nabla = 0$ and thus, observing also our standing assumption that $n \geq 3$,
	\[
	\begin{split}
		\|\Omega^Q \|_{L^{(n,1)}(\R^n)} \overset{\eqref{eq:Rnhodge}}{\aleq}& \|\underbrace{\Rz\cdot \Omega^Q}_{=0} \|_{L^{(n,1)}(\R^n)} +\|\Rz^\perp \Omega^Q \|_{L^{(n,1)}(\R^n)}\\
		\aleq&\|\Rz^\perp  (Q \nabla Q^T)\|_{L^{(n,1)}(\R^n)} + \|\Rz^\perp(Q\Omega Q^T)\|_{L^{(n,1)}(\R^n)}\\
		\aleq& \|[\Rz^\perp  ,Q] (\nabla Q^T)\|_{L^{(n,1)}(\R^n)} + \| [\Rz^\perp,Q]\Omega\|_{L^{(n,1)}(\R^n)} + \|\Rz^\perp \Omega \|_{L^{(n,1)}(\R^n)}\\
		\aleq&\|\nabla Q\|_{L^{(n,2)}(\R^n)}^2 + \|\g Q\|_{L^{(n,2)}(\R^n)}\|\Omega\|_{L^{(n,2)}(\R^n)} + \|\Rz^\perp \Omega\|_{L^{(n,1)}(\R^n)}.
	\end{split}
	\]
	In the last step we used \Cref{th:crwlorentz}.
\end{proof}

\begin{corollary}\label{co:reductiontoLn1}
	There exists $\gamma > 0$ such that the following holds. Consider the setting of \Cref{th:main1} : let $B\subset \R^n$ be a ball and assume that $u\in W^{1,n}(B;\R^N)$ solves the system
	\begin{align*}
		-\di\left( |\g u|^{n-2} \g u^i \right) &= \Omega_{ij} |\g u|^{n-2} \g u^j \ \text{ in }B,
	\end{align*}
	where $\Omega_{ij}\in L^n(B)$ satisfies $\Omega_{ij}=-\Omega_{ji}$ and 
	\begin{align*}
		\|\Omega \|_{L^{(n,2)}(B)} + \|\Rz^\perp \Omega\|_{L^{(n,1)}(\R^n)} &\leq \gamma.
	\end{align*}
	Then there exists $Q\in W^{1,n}(B;SO(N))$ such that the assumptions of \Cref{th:regularityLn1} are satisfied : $\|\g Q\|_{L^n(B)}\aleq \gamma$ and $u$ solves a system of the form
	\begin{align*}
		-\di\left( |\g u|^{n-2}Q \g u^i \right) &= \tilde{\Omega}_{ij} |\g u|^{n-2} \g u^j \ \text{ in }B,
	\end{align*}
	where $\|\tilde{\Omega}\|_{L^{(n,1)}(B)} \aleq \gamma$.
\end{corollary}
\begin{proof}
	We observe that we have
	\[
	- \div(Q |\nabla u|^{n-2} \nabla u) = \brac{-\nabla Q + Q\Omega } |\nabla u|^{n-2} \nabla u.
	\]
	We extend $\Omega$ to $\R^n$ by zero, and choose $Q$ from \Cref{pr:uhlenbeck}. Observe $(\nabla Q) Q^T = - Q \nabla Q^T$, so from \Cref{pr:uhlenbeckgauge} and the assumptions of \Cref{th:main1} we find
	\[
	\|\nabla Q + Q \Omega\|_{L^{(n,1)}} =\|\Omega_Q Q\|_{L^{(n,1)}} \aleq \gamma,
	\]
	and
	\[
	\|\nabla Q\|_{L^n(\R^n)} \aleq \gamma.
	\]
	Thus the assumptions of \Cref{th:regularityLn1} are satisfied.
\end{proof}

\section{Estimates below the natural exponent}\label{s:iwaniec}
Throughout this section, we denote $B_1$ the unit ball of $\R^n$ and $B(r)$ the ball with same center of $B_1$ and radius $r$.

Let $f\in L^1(B_1;\R^N) $ and $G\in L^\frac{n}{n-1}(B_1;\R^n\otimes \R^N)$. We assume that $u \in W^{1,n}(B_1;\R^N)$ solves
\[
\div(Q|\nabla u|^{n-2} \nabla u) = f + \di G \quad \text{in $B_1$},
\]
where $Q \in L^\infty(B_1,SO(N))$ and $\nabla Q \in L^n(B_1,\R^n \otimes \R^N)$.

In this section we prove theorems that are substantially motivated by the estimates obtained by Kuusi--Mingione \cite[Section 5]{KM18}. However, we were not able to use their techniques, in particular the $n$-harmonic approximation, because of the presence of $Q$. Instead we use the Iwaniec' stability result, \Cref{th:iwaniec}, to establish somewhat similar behavior.

The first result is a relatively standard consequence of Iwaniec' stability result, see e.g. \cite[Section~4.4]{MPS22}.
\begin{lemma}\label{la:easyiwaniec}
	Fix $\sigma \in (0,1)$ and $\theta \in (0,\frac{1}{4})$, There exists $\eps_0=\eps_0(n,N,\sigma,\theta) \in (0,1)$ and $\gamma_0=\gamma_0(n,N,\sigma) \in (0,1)$ such that the following holds:
	
	If $R<\frac{1}{2}$ and
	\[
	\|\nabla Q\|_{L^n(B(2R))} \leq \gamma_0
	\]
	and $u \in W^{1,n}(B(2R))$ satisfies
	\begin{equation}\label{eq:secondQnablauest}
		\div(Q|\nabla u|^{n-2} \nabla u) = f + \di G \text{ in $B(R)$}.
	\end{equation}
	Then for any $\eps \in (0,\eps_0)$
	\[
	\begin{split}
		(\theta R)^{-\eps}\|\nabla u\|_{L^{n-\eps}(B(\theta R))}^{n-\eps} \leq& \sigma\, R^{-\eps}\|\nabla u\|_{L^{n-\eps}(B(R))}^{n-\eps}\\
		&+C(\sigma) \brac{ R^{-\eps} \|\nabla u\|_{L^{n-\eps}(B(2R))}^{n-\eps} - (\theta R)^{-\eps}\|\nabla u\|_{L^{n-\eps}(B(\theta R))}^{n-\eps}}\\
		&+C(\eps,\sigma) \left( \|f\|_{L^1(B(R))}^{\frac{n-\eps}{n-1}} + \| G \|_{L^\frac{n}{n-1}(B(R))}^\frac{n-\eps}{n-1} \right)\\
	\end{split}
	\]
\end{lemma}
\begin{proof}
	Denote for $\rho > 0$ the cutoff functions $\eta_{\rho} \in C_c^\infty(B(2\rho))$, $\eta \equiv 1$ in $B(\rho)$, $|\eta| \leq 1$ in $\R^n$, and $|\nabla \eta_{\rho}| \aleq  \frac{1}{\rho}$.
	
	Let
	\[
	\tilde{u} := \eta_{\frac{R}{2}} (u-(u)_{B(R)}) + (u)_{B(R)}.
	\]
	Observe that $\supp \nabla \tilde{u} \subset B(R)$.
	
	And
	\[
	\tilde{Q} := \eta_{\frac{R}{2}} (Q-(Q)_{B(R)}) + (Q)_{B(R)}
	\]
	for which we observe that $[\tilde{Q}]_{BMO} \aleq \|\nabla Q\|_{L^n(B(R))}$.
	
	By Hodge decomposition on $\R^n$
	\[
	\tilde{Q}|\nabla \tilde{u}|^{-\eps} \nabla \tilde{u} = \nabla \varphi + B
	\]
	where for any $p \in (1,\infty)$, $q \in [1,\infty]$
	\[
	\|\nabla \varphi\|_{L^{(p,q)}(\R^n)} \aleq_{p,q} \|\tilde{Q}\|_{L^\infty}\, \||\nabla \tilde{u}|^{1-\eps}\|_{L^{(p,q)}(\R^n)},
	\]
	the constant here is independent of $\eps$; also by \Cref{th:crw} and \Cref{th:iwaniec} for all suitably small $\eps > 0$
	\begin{equation}\label{eq:estimateB}
		\|B\|_{L^{\frac{n-\eps}{1-\eps}}(\R^n)} \aleq \brac{|\eps|\|\tilde{Q}\|_{L^\infty}+[\tilde{Q}]_{BMO}} \|\nabla \tilde{u}\|_{L^{n-\eps}(\R^n)}^{1-\eps}
	\end{equation}
	We then have
	
	\begin{eqnarray*}
		\int_{\R^n} \eta_{\frac{R}{4}} |\nabla u|^{n-\eps} & = &\int_{\R^n} \eta_{\frac{R}{4}} Q|\nabla u|^{n-2}\nabla u: Q|\nabla u|^{-\eps} \nabla u\\
		& \overset{Q^T Q = I}{=}&\int Q|\nabla u|^{n-2}\nabla u: \eta_{\frac{R}{4}} \nabla \varphi\ +\int_{\R^n} \tilde{Q}|\nabla \tilde{u}|^{n-2}\nabla \tilde{u}: \eta_{\frac{R}{4}} B\\
		& = &\int_{\R^n} \tilde{Q}|\nabla \tilde{u}|^{n-2}\nabla \tilde{u}: \nabla (\eta_{\frac{R}{4}} \varphi) -\int_{\R^n} \tilde{Q}|\nabla \tilde{u}|^{n-2}\nabla \tilde{u}: \nabla \eta_{\frac{R}{4}}\, \varphi +\int_{\R^n} \tilde{Q}|\nabla \tilde{u}|^{n-2}\nabla \tilde{u}: \eta_{\frac{R}{4}} B\\
	\end{eqnarray*}
	
	Using \eqref{eq:secondQnablauest}
	\begin{equation}\label{eq:decompo_gu}
		\int_{\R^n} \eta_{\frac{R}{4}} |\nabla u|^{n-\eps} \leq  I + II + III+IV,
	\end{equation}
	where
	\[
	I = \abs{\int_{\R^n} f \eta_{\frac{R}{4}} \varphi},
	\]
	\[
	II = \abs{\int_{\R^n} Q|\nabla u|^{n-2}\nabla u: \nabla \eta_{\frac{R}{4}}\, \varphi},
	\]
	\[
	III = \abs{\int_{\R^n} Q|\nabla u|^{n-2}\nabla u: \eta_{\frac{R}{4}} B},
	\]
	\[
	IV = \abs{\int_{\R^n} G:\g(\eta_{\frac{R}{4}} \varphi) }.
	\]
	Fix $\delta > 0$. For $I$, we observe as before, for any $\eps > 0$
	
	\begin{align}
		I \aleq& \|f\|_{L^1(B(R))} \|\varphi\|_{L^\infty(\R^n)} \nonumber\\
		\aleq& \|f\|_{L^1(B(R))} \|\nabla \varphi\|_{L^{(n,1)}(\R^n)} \nonumber \\
		\aleq& \|f\|_{L^1(B(R))} \||\nabla \tilde{u}|^{1-\eps}\|_{L^{(n,1)}(B(R))} \nonumber\\
		\aleq& C(\eps)\|f\|_{L^1(B(R))} \||\nabla \tilde{u}|^{1-\eps}\|_{L^{\frac{n-\eps}{1-\eps}}(B(R))}  \nonumber\\
		\aleq&C(\eps) \|f\|_{L^1(B(R))} \, R^{\eps \frac{n-1}{n-\eps}} \|\nabla u\|_{L^{n-\eps}(B(R))}^{1-\eps}  \nonumber\\
		\aleq &\delta \|\nabla u\|_{L^{n-\eps}(B(R))}^{n-\eps} + C(\eps,\delta) R^\eps \brac{\|f\|_{L^1(B(R))}}^{\frac{n-\eps}{n-1}}. \label{eq:estimateI}
	\end{align}
	
	For $II$, it holds
	
	\begin{align}
		II\aleq& \|Q\|_{L^\infty}\, \frac{1}{R} \|\nabla u\|_{L^{n-\eps}\left(B(\frac{R}{2}) \setminus B(\frac{R}{4}) \right)}^{n-1}\, \|\varphi\|_{L^{\frac{n-\eps}{1-\eps}}(\R^n)} \nonumber\\
		\aleq&\frac{1}{R} \|\tilde{Q}\|_{L^\infty(\R^n)}\, \|\nabla u\|_{L^{n-\eps}\left(B(\frac{R}{2}) \setminus B(\frac{R}{4}) \right) }^{n-1}\, \|\nabla \varphi\|_{L^{\frac{n(n-\eps)}{n(1-\eps)+(n-\eps)}}(\R^n)} \nonumber\\
		\aleq&\frac{1}{R} \|\tilde{Q}\|_{L^\infty(\R^n)}\, \|\nabla u\|_{L^{n-\eps}\left(B(\frac{R}{2}) \setminus B(\frac{R}{4}) \right)}^{n-1}\, \|\nabla \tilde{u}\|_{L^{\frac{n(n-\eps)(1-\eps)}{n(1-\eps)+(n-\eps)}}(\R^n)}^{1-\eps}  \nonumber\\
		\aleq&\frac{1}{R} \|\tilde{Q}\|_{L^\infty(\R^n)}\, \|\nabla u\|_{L^{n-\eps}\left(B(\frac{R}{2}) \setminus B(\frac{R}{4}) \right)}^{n-1}\, \|\nabla u\|_{L^{\frac{n(n-\eps)(1-\eps)}{n(1-\eps)+(n-\eps)}}(B(R))}^{1-\eps}  \nonumber\\
		\aleq&\|\nabla u\|_{L^{n-\eps}\left(B(\frac{R}{2}) \setminus B(\frac{R}{4}) \right)}^{n-1}\, \|\nabla u\|_{L^{n-\eps}(B(R))}^{1-\eps}  \nonumber\\
		\aleq&\delta\|\nabla u\|_{L^{n-\eps}(B(R))}^{n-\eps}+C(\delta) \|\nabla u\|_{L^{n-\eps}\left(B(R) \setminus B(\frac{R}{4}) \right)}^{n-\eps}. \label{eq:estimateII}
	\end{align}
	
	Observe that the constants are independent of $\eps$ when $\eps$ is small, since $\frac{n(n-\eps)}{n(1-\eps)+(n-\eps)} \aeq \frac{n}{2}$. To estimate $III$, we use \eqref{eq:estimateB} :
	
	\begin{align}
		III\aleq &\|Q\|_{L^\infty} \|\eta_{\frac{R}{4}} \nabla u\|_{L^{n-\eps}(\R^n)}^{n-1} \|B\|_{L^{\frac{n-\eps}{1-\eps}}(\R^n)} \nonumber\\
		\aleq&\|\eta_{\frac{R}{4}} \nabla u\|_{L^{n-\eps}(\R^n)}^{n-1}  \brac{|\eps|\|\tilde{Q}\|_{L^\infty}+[Q]_{VMO}} \|\nabla \tilde{u}\|_{L^{n-\eps}(\R^n)}^{1-\eps}  \nonumber\\
		\aleq&\brac{\eps\|\tilde{Q}\|_{L^\infty}+[\tilde{Q}]_{VMO}}\, \|\nabla u\|_{L^{n-\eps}(B(R))}^{n-\eps}  \nonumber\\
		\aleq&\brac{\eps+\|\nabla Q\|_{L^n(B(R))}}\, \|\nabla u\|_{L^{n-\eps}(B(R))}^{n-\eps}. \label{eq:estimateIII}
	\end{align}

For $IV$, it holds
\begin{align}
	IV &\aleq \|G\|_{L^\frac{n}{n-1}(B(R))} \| \g (\eta_\frac{R}{4} \vp)\|_{L^n(B(R))} \nonumber\\
	&\aleq \|G\|_{L^\frac{n}{n-1}(B(R))} \left( \frac{1}{R}\| \vp\|_{L^n(B(R)\setminus B(\frac{R}{4}))} + \|\g\vp \|_{L^n(B(R))} \right)\nonumber\\
	&\aleq C(\eps)\|G\|_{L^\frac{n}{n-1}(B(R))} \left( \| \vp\|_{L^\infty(\R^n)} + \|\g\vp \|_{L^n(B(R))} \right) \nonumber\\
	&\aleq C(\eps)\|G\|_{L^\frac{n}{n-1}(B(R))} \left( \| \g \vp\|_{L^{(n,1)}(\R^n)} + \|\g\vp \|_{L^n(B(R))} \right) \nonumber\\
	&\aleq C(\eps)\|G\|_{L^\frac{n}{n-1}(B(R))} \left( \| |\g \tilde{u}|^{1-\eps} \|_{L^{(n,1)}(\R^n)} + \| |\g \tilde{u}|^{1-\eps} \|_{L^n(B(R))} \right) \nonumber\\
	&\aleq C(\eps)\|G\|_{L^\frac{n}{n-1}(B(R))} R^{\eps\frac{n-1}{n-\eps}} \|\g u \|_{L^{n-\eps}(B(R))}^{1-\eps} \nonumber\\
	&\aleq \delta\|\g u \|_{L^{n-\eps}(B(R))}^{1-\eps} + C(\eps,\delta) R^\eps\|G\|_{L^\frac{n}{n-1}(B(R))}^\frac{n-\eps}{n-1}. \label{eq:estimateIV}
\end{align}
	
	Thanks to \eqref{eq:decompo_gu}-\eqref{eq:estimateI}-\eqref{eq:estimateII}-\eqref{eq:estimateIII}-\eqref{eq:estimateIV}, we have shown
	\[
	\begin{split}
		(\theta R)^{-\eps}\|\nabla u\|_{L^{n-\eps}(B(\theta R))}^{n-\eps} \leq& \theta^{-\eps}2\delta\, R^{-\eps}\|\nabla \tilde{u}\|_{L^{n-\eps}(B(R))}^{n-\eps}\\
		&+C(\eps,\delta) \theta^{-\eps}\left( \|f\|_{L^1(B(R))}^{\frac{n-\eps}{n-1}} + \|G\|_{L^\frac{n}{n-1}(B(R))}^\frac{n-\eps}{n-1} \right)\\
		&+C\, \theta^{-\eps} \brac{ \eps+\|\nabla Q\|_{L^n(B(R))}}\, R^{-\eps} \|\nabla u\|_{L^{n-\eps}(B(R))}^{n-\eps}\\
		&+C(\delta) (\theta R)^{-\eps} \|\nabla u\|_{L^{n-\eps}\left( B(R)\setminus B(\frac{R}{4}) \right)}^{n-\eps}.\\
	\end{split}
	\]
	For the last term we observe
	
	\begin{align*}
		(\theta R)^{-\eps} \|\nabla u\|_{L^{n-\eps}(B(R)\setminus B(R/4))}^{n-\eps} \leq&(\theta R)^{-\eps} \|\nabla u\|_{L^{n-\eps}(B(R)\setminus B(\theta R))}^{n-\eps}\\
		=&(\theta R)^{-\eps} \|\nabla u\|_{L^{n-\eps}(B(R))}^{n-\eps} - (\theta R)^{-\eps}\|\nabla u\|_{L^{n-\eps}(B(\theta R))}^{n-\eps}\\
		=&R^{-\eps} \|\nabla u\|_{L^{n-\eps}(B(R))}^{n-\eps} - (\theta R)^{-\eps}\|\nabla u\|_{L^{n-\eps}(B(\theta R))}^{n-\eps} +(\theta^{-\eps} -1)   R^{-\eps}\|\nabla u\|_{L^{n-\eps}(B(R))}^{n-\eps}.
	\end{align*}
	So we arrive at
	\[
	\begin{split}
		(\theta R)^{-\eps}\|\nabla u\|_{L^{n-\eps}(B(\theta R))}^{n-\eps} \leq& \theta^{-\eps}2\delta\, R^{-\eps}\|\nabla \tilde{u}\|_{L^{n-\eps}(B(R))}^{n-\eps}\\
		&+ C(\delta) (\theta^{-\eps} -1)  \, R^{-\eps}\|\nabla u\|_{L^{n-\eps}(B(R))}^{n-\eps}\\
		&+C(\eps,\delta) \theta^{-\eps}\left( \|f\|_{L^1(B(R))}^{\frac{n-\eps}{n-1}} + \|G\|_{L^\frac{n}{n-1}(B(R))}^\frac{n-\eps}{n-1} \right)\\
		&+C\, \theta^{-\eps} \brac{ \eps+\|\nabla Q\|_{L^n(B(R))}}\, R^{-\eps} \|\nabla u\|_{L^{n-\eps}(B(R))}^{n-\eps}\\
		&+C(\delta) \brac{R^{-\eps} \|\nabla u\|_{L^{n-\eps}(B(R))}^{n-\eps} - (\theta R)^{-\eps}\|\nabla u\|_{L^{n-\eps}(B(\theta R))}^{n-\eps}}.
	\end{split}
	\]
	In a first step, for given $\sigma > 0$ we first choose $\delta > 0$ small so that
	\[
	\begin{split}
		(\theta R)^{-\eps}\|\nabla u\|_{L^{n-\eps}(B(\theta R))}^{n-\eps} \leq& \theta^{-\ve}\frac{\sigma}{8} R^{-\eps}\|\nabla \tilde{u}\|_{L^{n-\eps}(B(R))}^{n-\eps}\\
		&+ C(\sigma) (\theta^{-\ve}-1)  \, R^{-\eps}\|\nabla u\|_{L^{n-\eps}(B(R))}^{n-\eps}\\
		&+C(\eps,\sigma) \left( \|f\|_{L^1(B(R))}^{\frac{n-\eps}{n-1}} + \|G\|_{L^\frac{n}{n-1}(B(R))}^\frac{n-\eps}{n-1} \right)\\
		&+C \theta^{-\ve}\brac{\eps+\gamma_0}\, R^{-\eps} \|\nabla u\|_{L^{n-\eps}(B(R))}^{n-\eps}\\
		&+C(\sigma)\brac{R^{-\eps} \|\nabla u\|_{L^{n-\eps}(B(R))}^{n-\eps} - (\theta R)^{-\eps}\|\nabla u\|_{L^{n-\eps}(B(\theta R))}^{n-\eps}}.
	\end{split}
	\]
	Then we can assume that $\eps_0$ is small enough (depending on $\sigma$ and $\theta$) so that for any $\eps \in (0,\eps_0)$ we have $\theta^{-\eps} \leq 2$, $C\ve \leq \frac{\sigma}{16}$ and $\theta^{-\eps} -1 \leq \sqrt{\eps}$. This last condition is assured by the asymptotic expansion $e^{-\ve\log(\theta)}-1 \underset{\ve\to 0}{\sim} - \ve \log(\theta)$. Then we have for any $\eps \in (0,\eps_0)$
	\[
	\begin{split}
		(\theta R)^{-\eps}\|\nabla u\|_{L^{n-\eps}(B(\theta R))}^{n-\eps} \leq& \frac{\sigma}{4} R^{-\eps}\|\nabla \tilde{u}\|_{L^{n-\eps}(B(R))}^{n-\eps}\\
		&+ C(\sigma) \sqrt{\eps}  \, R^{-\eps}\|\nabla u\|_{L^{n-\eps}(B(R))}^{n-\eps}\\
		&+C(\eps,\sigma) \left( \|f\|_{L^1(B(R))}^{\frac{n-\eps}{n-1}} + \|G\|_{L^\frac{n}{n-1}(B(R))}^\frac{n-\eps}{n-1}\right) \\
		&+ \brac{\frac{\sigma}{8}+2C\gamma_0}\, R^{-\eps} \|\nabla u\|_{L^{n-\eps}(B(R))}^{n-\eps}\\
		&+C(\sigma) \brac{R^{-\eps} \|\nabla u\|_{L^{n-\eps}(B(R))}^{n-\eps} - (\theta R)^{-\eps}\|\nabla u\|_{L^{n-\eps}(B(\theta R))}^{n-\eps}}.
	\end{split}
	\]
	Next we reduce again $\eps_0$ so that for any $\eps \in (0,\eps_0)$ we can absorb $C(\sigma)$, and we choose $\gamma_0>0$ such that $2C\gamma_0 \leq \frac{\sigma}{1000}$:
	\[
	\begin{split}
		(\theta R)^{-\eps}\|\nabla u\|_{L^{n-\eps}(B(\theta R))}^{n-\eps} \leq& \frac{\sigma}{2} R^{-\eps}\|\nabla \tilde{u}\|_{L^{n-\eps}(B(R))}^{n-\eps}\\
		&+C(\eps,\sigma) \left( \|f\|_{L^1(B(R))}^{\frac{n-\eps}{n-1}} + \|G\|_{L^\frac{n}{n-1}(B(R))}^\frac{n-\eps}{n-1} \right)\\
		&+C(\sigma)\brac{R^{-\eps} \|\nabla u\|_{L^{n-\eps}(B(R))}^{n-\eps} - (\theta R)^{-\eps}\|\nabla u\|_{L^{n-\eps}(B(\theta R))}^{n-\eps}}.
	\end{split}
	\]
\end{proof}

\Cref{la:easyiwaniec} is not enough to conclude regularity for $n$-harmonic map systems or similar equation because $\|f\|_{L^1}$ decays in terms of $\|\nabla u\|_{L^n}$, not $\|\nabla u\|_{L^{n-\eps}}$. This is the main reason for all the additional differentiability assumptions in the literature, cf. \cite{SS17}.

The following result is motivated by the iterative estimates as in \cite[Lemma 5.2 or Proposition 5.1]{KM18}. Kuusi--Mingione developed a $n$-harmonic approximation theory for this \cite[Section 4]{KM18}-- which we were note able to reproduce for our operator $\div(Q |\nabla u|^{n-2} \nabla u)$ because of the rotation $Q$. Instead, we turn to Iwaniec' stability theorem to obtain comparable estimates without harmonic approximation -- which works because we consider $W^{1,p}$ for $p \approx n$, so Iwaniec' stability allows us to move into continuous test-functions with controlled expenses in error terms.

\begin{lemma}\label{lm:distance_harm}
	Fix $\sigma \in (0,1)$. There exists $\eps_1=\eps_1(n,N,\sigma) > 0$ such that the following holds.
	
	For any $\eps \in (0,\eps_1)$ there exists $\gamma_1 = \gamma_1(n,N,\eps,\sigma) > 0$ with the following properties.
	
	Assume $u \in W^{1,n}(B(4/3),\R^N)$, $f \in L^1(B(1),\R^N)$, $G\in L^\frac{n}{n-1}(B(1),\R^n\otimes\R^N)$, $Q \in W^{1,n}(B(4/3),SO(N))$ are solutions to
	\[
	\div(Q|\nabla u|^{n-2} \nabla u) = f + \di G\quad \text{in $B(1)$}.
	\]
	Assume moreover
	\begin{itemize} \item $\partial B(1)$ is a good slice in the sense that
		\begin{equation}\label{eq:goodslice}
			[u]_{W^{1-\frac{1}{n},n}(\partial B(1))} \aleq [u]_{W^{1-\frac{1}{n},n}(B(4/3))} \aleq \|\nabla u\|_{L^{\frac{n^2}{n+1}}(B(4/3))}.
		\end{equation}
		\item $v$ solves
		\[
		\begin{cases}
			\lap_n v = 0 \quad &\text{in $B(1)$},\\
			v =u \quad &\text{on $\partial B(1)$}.
		\end{cases}
		\]
		\item $\|\nabla Q\|_{L^n(B(4/3))} \leq \gamma_1$.
	\end{itemize}
	Then we have the estimate
	\[
	\begin{split}
		\|\nabla u - \nabla v\|_{L^{n-\eps}(B(1))}  \aleq&C(\sigma,\eps) \left( \|f\|_{L^1(B(1))}^{\frac{1}{n-1}} + \| G\|_{L^\frac{n}{n-1}(B(1))}^\frac{1}{n-1} \right) +\sigma \|\nabla u\|_{L^{n-\eps}(B(4/3))} .
	\end{split}
	\]
\end{lemma}

\begin{proof}[Proof of \Cref{lm:distance_harm}]
	We recall the following well-known inequality (cf. e.g. \cite{B05}, $:$ denotes the Hilbert-Schmidt scalar product)
	\[
	|X-Y|^p \leq C_p \brac{|X|^{p-2} X - |Y|^{p-2} Y}: \brac{X-Y}, \quad  \forall X,Y  \in \R^{n \times N}, \quad p \in [2,\infty).
	\]
	Then we have, using $Q^T Q = I$
	\[
	\begin{split}
		&\int_{B(1)} |\nabla u - \nabla v|^{n-\eps}\\
		\aleq&   \int_{B(1)} \brac{|\nabla u|^{n-2} \nabla v^i - |\nabla v|^{n-2} \nabla v^i} :|\nabla (u-v)|^{-\eps} \nabla (u-v)^i\\
		=&   \int_{B(1)} \brac{Q_{ik} |\nabla u|^{n-2} \nabla u^k - Q_{ik}|\nabla v|^{n-2} \nabla v^k}: Q_{ij} |\nabla (u-v)|^{-\eps} \nabla (u-v)^j.
	\end{split}
	\]
	Since $u-v \in W^{1,n}_0(B(1))$ we find from Hodge decomposition, \Cref{co:ourhodge}, $a^i \in W^{1,\frac{n-\eps}{1-\eps}}_0(B(1))$ and $B^i \in L^{\frac{n-\eps}{1-\eps}}(B(1),\R^n)$ such that
	\[
	Q_{ij} |\nabla (u-v)|^{-\eps} d(u-v)^j = \nabla a^i + B^i \quad \text{in $B(1)$},
	\]
	and the estimates (choosing in \Cref{co:ourhodge} the extension $\tilde{Q} := \eta (Q-(Q)_{B(4/3)}) + (Q)_{B(4/3)}$ for some cutoff function $\eta \in C_c^\infty(B(4/3))$, $\eta \equiv 1$ in $B(1)$)
	\begin{equation}\label{eq:diff:1}
		\|a\|_{W^{1,\frac{n-\eps}{1-\eps}}(B(1))} \aleq \|Q\|_{L^\infty(B(4/3))} \|\nabla (u-v)\|_{L^{n-\eps}(B(1))}^{1-\eps},
	\end{equation}
	and
	\begin{equation}\label{eq:diff:2}
		\|B\|_{L^{\frac{n-\eps}{1-\eps}}(B(1))} \aleq \brac{\eps\, \|Q\|_{L^\infty(B(4/3))}+\|\nabla Q\|_{L^n(B(4/3))}} \|\nabla (u-v)\|_{L^{n-\eps}(B(1))}^{1-\eps}.
	\end{equation}
	We now drop the indices for better readability, and have
	\begin{equation}\label{eq:dif:dec}
		\int_{B(1)} |\nabla u - \nabla v|^{n-\eps} \aleq I + II + III + IV
	\end{equation}
	where
	\[
	\begin{split}
		I :=&\abs{\int_{B(1)} Q|\nabla u|^{n-2} \nabla u: \nabla a} \\
		II := &\abs{\int_{B(1)} Q|\nabla v|^{n-2} \nabla v: \nabla a} \\
		III := & \abs{\int_{B(1)} Q|\nabla u|^{n-2} \nabla u:B} \\
		IV := &\abs{\int_{B(1)} Q|\nabla v|^{n-2} \nabla v: B}.
	\end{split}
	\]
	
	Fix now some small $\delta > 0$.
	
	Since $a$ has zero-boundary data, and by Sobolev embedding, and then H\"older inequality (observe that $\frac{n-\eps}{1-\eps} > n$),
	\[
	\begin{split}
		I = \abs{\int_{B(1)} f a + G:\g a} \aleq & \|f\|_{L^1(B(1))}\, \|a\|_{L^\infty(B(1))} + \|G\|_{L^\frac{n}{n-1}(B(1))} \|\g a \|_{L^n(B(1))}\\
		\aleq &\|f\|_{L^1(B(1))}\, \|\nabla a\|_{L^{(n,1)}(B(1))} + \|G\|_{L^\frac{n}{n-1}(B(1))} \|\g a \|_{L^n(B(1))}\\
		\aleq &C(\eps) \left( \|f\|_{L^1(B(1))} + \|G\|_{L^\frac{n}{n-1}(B(1))} \right) \|\nabla a\|_{L^{\frac{n-\eps}{1-\eps}}(B(1))}\\
		\overset{\eqref{eq:diff:1}}{\aleq}& C(\eps)  \left( \|f\|_{L^1(B(1))} + \|G\|_{L^\frac{n}{n-1}(B(1))} \right) \|\nabla (u-v)\|_{L^{n-\eps}(B(1))}^{1-\eps}.
	\end{split}
	\]
	That is, by Young's inequality
	\[
	I \leq \delta \|\nabla (u-v)\|_{L^{n-\eps}(B(1))}^{n-\eps} +  C(\delta,\eps)  \left( \|f\|_{L^1(B(1))}^{\frac{n-\eps}{n-1}} + \|G\|_{L^\frac{n}{n-1}(B(1))}^{\frac{n-\eps}{n-1}} \right).
	\]

	As for $II$, observe that $v$ is $n$-harmonic. Thus $\nabla v$ is $L^n$-norm minimizing, and thus the usual harmonic extension $u^h$ of $u \Big |_{\partial B(1)}$ is a competitor, that is from trace theorem and H\"older's inequality,
	\begin{equation}\label{eq:harmonicv}
		\begin{split}
			\|\nabla v\|_{L^{n}(B(1))} \leq &\|\nabla u^h\|_{L^{n}(B(1))} \aleq [u]_{W^{1-\frac{1}{n},n}(\partial B(1))} \overset{\eqref{eq:goodslice}}{\aleq} \|\nabla u\|_{L^{\frac{n^2}{n+1}}(B(4/3))}
			\overset{\eps \ll 1}{\aleq}\|\nabla u\|_{L^{n-\eps}(B(4/3))}
		\end{split}
	\end{equation}
	Also observe that
	\[
	\div(Q |\nabla v|^{n-2} \nabla v) = \nabla Q |\nabla v|^{n-2} \nabla v \quad \text{in $B(1)$},
	\]
	so that we find
	\[
	\begin{split}
		II \aleq & \|\nabla Q\|_{L^n(B(1))} \|\nabla v\|_{L^n(B(1))}^{n-1}\, \|a\|_{L^\infty(B(1))} \aleq C(\eps)  \|\nabla Q\|_{L^n(B(1))}\|\nabla u\|_{L^{n-\eps}(B(4/3))}^{n-1} \|\nabla (u-v)\|_{L^{n-\eps}(B(1))}^{1-\eps}.
	\end{split}
	\]
	Consequently,
	\[
	II \aleq \delta \|\nabla (u-v)\|_{L^{n-\eps}(B(1))}^{n-\eps} + C(\eps,\delta) \|\nabla Q\|_{L^n(B(1))}^{\frac{n-\eps}{n-1}}\, \|\nabla u\|_{L^{n-\eps}(B(4/3))}^{n-\eps}.
	\]
	From \eqref{eq:diff:2} we can estimate
	\[
	III \aleq \|\nabla u\|_{L^{n-\eps}(B(1))}^{n-1} \|B\|_{L^{\frac{n-\eps}{1-\eps}}(B(1))} \aleq\brac{\eps +\|\nabla Q\|_{L^n(B(4/3))}} \|\nabla (u-v)\|_{L^{n-\eps}(B(1))}^{n-\eps}
	\]
	With \eqref{eq:harmonicv} we conclude
	\[
	III \aleq \brac{\eps +\|\nabla Q\|_{L^n(B(4/3))}} \|\nabla u\|_{L^{n-\eps}(B(1))}^{n-\eps}.
	\]
	Lastly, with the help of H\"older's inequality, \eqref{eq:harmonicv}, and \eqref{eq:diff:2}
	\[
	IV \aleq \|\nabla v\|_{L^{n}(B(1))}^{n-1}\, \|B\|_{L^{\frac{n-\eps}{1-\eps}}(B(1))} \aleq \|\nabla u\|_{L^{n-\eps}(B(1))}^{n-1} \brac{\eps +\|\nabla Q\|_{L^n(B(4/3))}} \|\nabla (u-v)\|_{L^{n-\eps}(B(1))}^{1-\eps}
	\]
	Thus, again by Young's inequality we have
	\[
	IV \leq \delta \|\nabla (u-v)\|_{L^{n-\eps}(B(1))}^{n-\eps} + C(\delta)\brac{\eps +\|\nabla Q\|_{L^n(B(4/3))}}^{\frac{n-\eps}{n-1}}  \|\nabla u\|_{L^{n-\eps}(B(1))}^{n-\eps}
	\]
	Plugging all these estimates together we obtain from \eqref{eq:dif:dec} and using the smallness of $\|\nabla Q\|_{L^n(B(4/3))}$,
	\[
	\begin{split}
		\|\nabla (u-v)\|_{L^{n-\eps}(B(1))}^{n-\eps} \leq& 3\delta \|\nabla (u-v)\|_{L^{n-\eps}(B(1))}^{n-\eps}\\
		&+  C(\delta,\eps)  \left( \|f\|_{L^1(B(1))}^{\frac{n-\eps}{n-1}} + \|G\|_{L^\frac{n}{n-1}(B(1))}^{\frac{n-\eps}{n-1}} \right)\\
		&+ C(\eps,\delta) \gamma_1^{\frac{n-\eps}{n-1}}\, \|\nabla u\|_{L^{n-\eps}(B(4/3))}^{n-\eps}\\
		&+\brac{\eps +\gamma_1} \|\nabla u\|_{L^{n-\eps}(B(1))}^{n-\eps}\\
		&+ C(\delta)\brac{\eps+\gamma_1}^{\frac{n-\eps}{n-1}}  \|\nabla u\|_{L^{n-\eps}(B(1))}^{n-\eps}.
	\end{split}
	\]
	Choosing $\delta$ sufficiently small, we can absorb the first term on the right-hand side and have found
	\[
	\begin{split}
		\|\nabla (u-v)\|_{L^{n-\eps}(B(1))}^{n-\eps} \leq&   C(\eps)  \left( \|f\|_{L^1(B(1))}^{\frac{n-\eps}{n-1}} + \|G\|_{L^\frac{n}{n-1}(B(1))}^{\frac{n-\eps}{n-1}} \right)\\
		&+ C(\eps) \gamma^{\frac{n-\eps}{n-1}}\, \|\nabla u\|_{L^{n-\eps}(B(4/3))}^{n-\eps}\\
		&+\brac{\eps +\gamma} \|\nabla u\|_{L^{n-\eps}(B(1))}^{n-\eps}\\
		&+ \brac{\eps+\gamma}^{\frac{n-\eps}{n-1}}  \|\nabla u\|_{L^{n-\eps}(B(1))}^{n-\eps}.
	\end{split}
	\]
	Now we can choose $\eps_1$ small enough so that for any $\eps \in (0,\eps_1)$ we have
	\[
	\begin{split}
		\|\nabla (u-v)\|_{L^{n-\eps}(B(1))}^{n-\eps} \leq&   C(\eps)  \left( \|f\|_{L^1(B(1))}^{\frac{n-\eps}{n-1}} + \|G\|_{L^\frac{n}{n-1}(B(1))}^{\frac{n-\eps}{n-1}} \right)\\
		&+ C(\eps) \gamma^{\frac{n-\eps}{n-1}}\, \|\nabla u\|_{L^{n-\eps}(B(4/3))}^{n-\eps}\\
		&+\gamma \|\nabla u\|_{L^{n-\eps}(B(1))}^{n-\eps}\\
		&+ \gamma^{\frac{n-\eps}{n-1}}  \|\nabla u\|_{L^{n-\eps}(B(1))}^{n-\eps}\\
		& + \left( \frac{\sigma}{10} + \brac{\frac{\sigma}{4}}^{n-\eps} \right) \|\nabla u\|_{L^{n-\eps}(B(1))}^{n-\eps}.
	\end{split}
	\]
	For any $\eps \in (0,\eps_0)$ we can then choose a small $\gamma$ that compensates for $C(\eps)$ in the second line, so that
	\[
	\begin{split}
		\|\nabla (u-v)\|_{L^{n-\eps}(B(1))}^{n-\eps} \leq&   C(\eps)  \left( \|f\|_{L^1(B(1))}^{\frac{n-\eps}{n-1}} + \|G\|_{L^\frac{n}{n-1}(B(1))}^{\frac{n-\eps}{n-1}} \right) +\left( \frac{\sigma}{2} + \brac{\frac{\sigma}{2}}^{n-\eps} \right) \|\nabla u\|_{L^{n-\eps}(B(4/3))}^{n-\eps}.
	\end{split}
	\]
	We can conclude.

\end{proof}

\section{Lorentz-space estimates for n-Laplacian systems: Proof of Proposition~\ref{Lninfty_estimateintro}}\label{s:lorentz}
The goal of this section is to prove the following result which is probably interesting on its own
\begin{proposition}\label{Lninfty_estimate}
	There are some uniform constants $\Gamma=\Gamma(n,N)>1$ and $\tau_1 =\tau_1(n,N)\in(0,\frac{1}{4})$  and $\alpha = \alpha(n,N) \in (0,1)$ satisfying the following.
	
	For any $\sigma\in(0,1)$, there exists  $\ve_0 = \ve_0(n,N,\sigma) \in(0,1)$ such that for any $\ve \in (0,\ve_0)$ we find $\gamma_0=\gamma_0(n,N,\ve,\sigma) \in (0,1)$ such that the following holds.
	
	Assume $f\in L^1(B^n(0,1);\R^N)$, $G\in L^\frac{n}{n-1}(B(0,1);\R^n \otimes \R^N)$, $Q\in W^{1,n}(B(0,1);SO(N))$ and $u \in W^{1,n}(B(0,1);\R^N)$ satisfy the system
	\begin{align*}
		-\di(|\g u|^{n-2} Q\g u) = f + \di G\ \text{in }B(0,1)
	\end{align*}
	and the bound
	\[
	\|\g Q\|_{L^n(B(0,1))} \leq \gamma_0.
	\]
	For any $\tau\in(0,\tau_1)$, there exists $C_1=C_1(n,N,\sigma,\ve,\tau)>0$ and $C_2 = C_2(n,N,\tau)>0$ such that
	\begin{align*}
		\| \g u \|_{L^{(n,\infty)}(B(0,\tau))} &\leq C_1(\sigma,\ve,\tau) \left( \|f\|_{L^1(B(0,1))}^\frac{1}{n-1} + \|G\|_{L^\frac{n}{n-1}(B(0,1))}^\frac{1}{n-1}\right) \\
		&+ C_2(\tau) \|\g u \|_{L^{n-\ve}(B(0,1))} + \Gamma \left( \frac{\sigma }{\tau^\frac{n}{n-1} } + \tau^\alpha \right) \| \g u \|_{L^{(n,\infty)}(B(0,1))}.
	\end{align*}
\end{proposition}

We first list our main ingredients:
We first record the scaled version of \Cref{lm:distance_harm}.
\begin{corollary}\label{co:initial_estimate_difference}
	Let $\sigma\in(0,1)$ and consider $\ve_1=\ve_1(n,N,\sigma)\in(0,1)$ be given by \Cref{lm:distance_harm}. Let $\ve \in (0,\ve_1)$ and $\gamma_1 = \gamma_1(n,N,\ve,\sigma)>0$ also given by \Cref{lm:distance_harm}. There exists $C_0=C_0(n,N,\sigma,\eps)>0$ such that the following hold.
	
	Assume that $u \in W^{1,n}(B(x,r);\R^N)$ satisfies
	\[
	\div(Q|\nabla u|^{n-2} \nabla u) = f +\di G\quad \text{in $B(x,r)$}.
	\]
	where $Q \in W^{1,n}(B(x,r);SO(N))$ and
	\begin{align}\label{eq:assumption_Q}
		\|\nabla Q\|_{L^n(B(x,r))} \leq \gamma_1.
	\end{align}
	
	There exists a radius $\rho\in\left[ \frac{1}{2}r,\frac{3}{4}r \right]$ such that if $v\in W^{1,n}(B(x,\rho);\R^N)$ satisfies
	\begin{align*}
		\left\{ \begin{array}{c l}
			\lap_n v = 0 & \text{in }B(x,\rho),\\
			v = u & \text{on }\dr B(x,\rho),
		\end{array}
		\right.
	\end{align*}
	then it holds
	\begin{align}\label{initial_estimate_difference}
		\left( \mvint_{B(x,\rho)} |\g u - \g v|^{n-\ve} \right)^\frac{1}{n-\ve} &\leq C_0(\sigma,\eps)\left( \frac{1}{r^{n-1}} \int_{B(x,r)} |f| \right)^\frac{1}{n-1}  + C_0(\sigma,\eps)\left(\mvint_{B(x,r)} |G|^\frac{n}{n-1}\right)^\frac{1}{n}\\
		&+ \sigma \left( \mvint_{B(x,r)} |\g u|^{n-\ve} \right)^\frac{1}{n-\ve}.
	\end{align}
	
\end{corollary}

The following regularity result for $n$-harmonic maps into $\R^N$ was proven in \cite[Theorem 3.2]{KM18}.
\begin{lemma}\label{KM:theorem32}
	There exists $\tau_0=\tau_0(n,N)\in(0,\frac{1}{2})$, $c_{hol}=c_{hol}(n,N)>0$ and $\alpha=\alpha(n,N) \in (0,1)$ such that the following holds. Assume $v\in W^{1,n}(B(x,\rho))$ satisfies
	\[
	\lap_n v = 0 \quad \text{in $B(x,\rho)$}
	\]
	Then for any $\tau \in(0,\tau_0)$ and any ball $B(y,s)\subset B(x,\rho)$ we have
	\begin{align}\label{oscillation_gv}
		\osc_{B(y,\tau s)} (\g v) &\leq c_{hol} \tau^\alpha \mvint_{B(y,s)} |\g v - (\g v)_{B(y,s)}|.
	\end{align}
\end{lemma}

Using this Hölder-continuity estimate, we can exchange in \eqref{initial_estimate_difference} the $n$-harmonic map $\g v$  with the average $(\g u)_{B(x,\tau r)}$. Namely we have,

\begin{lemma}\label{rewrite_initial_estimate}
	There exists $\tau_0=\tau_0(n,N)\in(0,\frac{1}{2})$, $\Gamma=\Gamma(n,N)>1$ such that the following holds.
	
	Fix $\sigma \in (0,1)$, take $\eps_1$ from \Cref{co:initial_estimate_difference}. For $\eps \in (0,\eps_1)$ take $\gamma_1$ also from \Cref{co:initial_estimate_difference} and assume \eqref{eq:assumption_Q}.
	
	For any $\tau\in(0,\tau_0)$ and $\ve\in(0,\ve_1)$, there exists $C_1=C_1(n,N,\sigma,\tau,\eps)>1$
	such that for any ball $B(x,r)\subset B(0,1)$ the following holds when $u,Q,f,G$ are as in \Cref{Lninfty_estimate}:
	\begin{align*}
		\left( \mvint_{B(x,\tau r)} |\g u - (\g u)_{B(x,\sigma r)}|^{n-\ve} \right)^\frac{1}{n-\ve} &\leq C_1(\sigma,\tau,\eps) \left( \frac{1}{r^{n-1}} \int_{B(x,r)} |f| \right)^\frac{1}{n-1} + C_1(\sigma,\tau,\eps) \left(\mvint_{B(x,r)} |G|^\frac{n}{n-1}\right)^\frac{1}{n}\\
		& +\Gamma \left( \frac{\sigma}{\tau^\frac{n}{n-1}} + \tau^\alpha \sigma \right)\left( \mvint_{B(x,r)} |\g u|^{n-\ve} \right)^\frac{1}{n-\ve} \\
		&+  \Gamma \tau^\alpha \left( \mvint_{B(x,r)} |\g u - (\g u)_{B(x,r)}|^{n-\ve} \right)^\frac{1}{n-\ve}\\
	\end{align*}
\end{lemma}

\begin{proof}
	We take $\tau_0$ to be $\frac{1}{100}$ of the $\tau_0$ from \Cref{KM:theorem32}.
	
	Pick $\rho \in (\frac{r}{2},\frac{3}{4}r)$ from \Cref{co:initial_estimate_difference} and let $v\in W^{1,n}(B(x,\rho);\R^N)$ satisfy
	\begin{align*}
		\left\{ \begin{array}{c l}
			\lap_n v = 0 & \text{in }B(x,\rho),\\
			v = u & \text{on }\dr B(x,\rho),
		\end{array}
		\right.
	\end{align*}
	
	By triangle inequality
	\begin{align*}
		\left( \mvint_{B(x,\tau r)} |\g u - (\g u)_{B(x,\tau r)}|^{n-\ve} \right)^\frac{1}{n-\ve} &\leq 2\left( \mvint_{B(x,\tau r)} |\g u - \g v|^{n-\ve}  \right)^\frac{1}{n-\ve} + \osc_{B(x,\tau r)} (\g v).
	\end{align*}
	
	For $\tau < \tau_0$ we use \eqref{oscillation_gv} we obtain
	\begin{align*}
		\left( \mvint_{B(x,\tau r)} |\g u - (\g u)_{B(x, \tau r)}|^{n-\ve} \right)^\frac{1}{n-\ve} &\leq  2\left(\frac{1}{2^n \tau^n} \mvint_{B(x,r/2)} |\g u - \g v|^{n-\ve}  \right)^\frac{1}{n-\ve}  \\
		&+ c_{hol}\tau^\alpha \left( \mvint_{B(x,r/2)} |\g v - (\g v)_{B(x,r/2)}|^{n-\ve} \right)^\frac{1}{n-\ve}\\
		&\aleq_n  \left(\frac{1}{\tau^n} \mvint_{B(x,\rho)} |\g u - \g v|^{n-\ve}  \right)^\frac{1}{n-\ve}  \\
		&+ c_{hol}\tau^\alpha \left( \mvint_{B(x,\rho)} |\g v - (\g v)_{B(x,\rho)}|^{n-\ve} \right)^\frac{1}{n-\ve}.
	\end{align*}
	We estimate the first term thanks to \eqref{initial_estimate_difference} and the second to last term by the triangle inequality
	\begin{align*}
		\left( \mvint_{B(x, \tau r)} |\g u - (\g u)_{B(x, \tau r)}|^{n-\ve} \right)^\frac{1}{n-\ve} &\aleq_n \frac{C_0(\sigma,\eps)}{\tau^\frac{n}{n-\ve}}\left( \frac{1}{r^{n-1}} \int_{B(x,r)} |f| \right)^\frac{1}{n-1} + \frac{C_0(\sigma,\eps)}{\tau^\frac{n}{n-\eps}} \left( \mvint_{B(x,r)} |G|^\frac{n}{n-1}\right)^\frac{1}{n} \\
		&+ \frac{\sigma}{\tau^\frac{n}{n-\ve}} \left( \mvint_{B(x,r)} |\g u|^{n-\ve} \right)^\frac{1}{n-\ve} \\
		&+ 2c_{hol} \tau^\alpha \left( \mvint_{B(x,\rho)} |\g u - \g v|^{n-\ve} \right)^\frac{1}{n-\ve} + c_{hol} \tau^\alpha \left( \mvint_{B(x,\rho)} |\g u - (\g u)_{B(x,\rho)}|^{n-\ve} \right)^\frac{1}{n-\ve}.
	\end{align*}
	Using again \eqref{initial_estimate_difference} for the last term above
	\begin{align*}
		\left( \mvint_{B(x,\tau r)} |\g u - (\g u)_{B(x,\tau r)}|^{n-\ve} \right)^\frac{1}{n-\ve} &\aleq_n \left( \frac{C_0(\sigma,\eps)}{\tau^\frac{n}{n-\ve}}+ \tau^\alpha c_0 c_{hol} \right)\left( \frac{1}{r^{n-1}} \int_{B(x,r)} |f| \right)^\frac{1}{n-1} \\
		&+  \left( \frac{C_0(\sigma,\eps)}{\tau^\frac{n}{n-\ve}}+ \tau^\alpha c_0 c_{hol} \right) \left( \mvint_{B(x,r)} |G|^\frac{n}{n-1}\right)^\frac{1}{n}\\
		&+ \left( \frac{2\sigma}{\tau^\frac{n}{n-\ve}} + 2\tau^\alpha c_{hol} \sigma \right) \left( \mvint_{B(x,r)} |\g u|^{n-\ve} \right)^\frac{1}{n-\ve} \\
		&+  c_{hol} \tau^\alpha \left( \mvint_{B(x,r)} |\g u - (\g u)_{B(x,r)}|^{n-\ve} \right)^\frac{1}{n-\ve}.
	\end{align*}
	We conclude by observing that since $\tau\in(0,1)$ and $\ve\in(0,1)$, we have the inequality $\tau^{-\frac{n}{n-\ve}} \leq \tau^{-\frac{n}{n-1}}$.
\end{proof}

In order to obtain $L^{(n,\infty)}$-estimates, we will rewrite the estimate of \Cref{rewrite_initial_estimate} in the form of maximal functions: For a given $x\in B(0,1)$ we take the supremum over $r>0$ in the estimates and then consider the result as an estimate between maximal functions.

We consider the following maximal functions: given $p,q\in[1,\infty)$, $\tau\in(0,1)$ and a function $g\in L^1(B(0,1))$, we set
\begin{align*}
	M^\sharp_{\tau,p} g(x) &= \sup\left\{ \left( \mvint_{B(x,r)} |g - (g)_{B(x,r)}|^p \right)^\frac{1}{p} : B(x,r/\tau) \subset B(0,1) \right\},\\
	M^\sharp_{p} g(x) &= \sup\left\{ \left( \mvint_{B(x,r)} |g - (g)_{B(x,r)}|^{p} \right)^\frac{1}{p} : B(x,r) \subset B(0,1) \right\},\\
	M_q g(x) &= \sup\left\{ \left( \frac{1}{ r^{n-q} }\int_{B(x,r)} |g|^q \right)^\frac{1}{q} : B(x,r) \subset B(0,1) \right\},\\
	M g(x) &= \sup\left\{ \mvint_{B(x,r)} |g| : B(x,r) \subset B(0,1) \right\}.
\end{align*}

\begin{lemma}\label{estimate_norm_MsGu}
	There are universal constants $\tau_1 = \tau_1(n,N) \in(0, \tau_0)$ and $\Gamma = \Gamma(n,N)>1$ satisfying the following. Let $\sigma\in(0,1)$ and consider $\ve_1 = \ve_1(n,N,\sigma)$ given by \Cref{co:initial_estimate_difference}. Let $\ve \in(0,\ve_1)$ and assume \eqref{eq:assumption_Q}. There exists $c_1=c_1(n,N,\sigma,\tau,\ve)>0$ and $c_2 = c_2(n,N,\tau)>0$ such that
	\begin{align*}
		\left\| M^\sharp_{\tau,n-\ve}(\g u) \right\|_{L^{(n,\infty)}\left( B(0,\frac{1}{2}) \right)} &\leq c_1(\sigma,\tau,\ve) \left( \|f\|_{L^1(B(0,1))}^\frac{1}{n-1} + \|G\|_{L^\frac{n}{n-1}(B(0,1))}^\frac{1}{n-1}\right)\\
		& +c_2(\tau) \|\g u \|_{L^{n-\ve}(B(0,1))} + \Gamma \left( \frac{\sigma}{\tau^\frac{n}{n-1}} + \tau^\alpha \sigma \right) \| \g u \|_{L^{(n,\infty)}(B(0,1))}.
	\end{align*}
\end{lemma}

\begin{proof}
	Thanks to \Cref{rewrite_initial_estimate}, we obtain the pointwise estimate in $B(0,1/2)$ :
	\begin{align*}
		M^\sharp_{\tau,n-\ve}(\g u) &\leq c_1 (M_1 f)^\frac{1}{n-1}+c_1 \brac{M\left[|G|^\frac{n}{n-1}\right]}^\frac{1}{n} + \Gamma \left( \frac{\sigma}{\tau^\frac{n}{n-1}} + \tau^\alpha \sigma \right) M[|\g u|^{n-\ve}]^\frac{1}{n-\ve} + \Gamma \tau^\alpha M^\sharp_{n-\ve}(\g u)
	\end{align*}
	So for any $\lambda>0$, we can estimate the level sets of $M^\sharp_{\tau,n-\ve}(\g u)$ on $B(0,\frac{1}{2})$ :
	\begin{align}
		\left| \left\{ x \in B\left( 0, \frac{1}{2} \right) : M^\sharp_{\tau,n-\ve}(\g u)(x) > \lambda \right\} \right| &\leq \left| \left\{ x \in B\left( 0, \frac{1}{2} \right) : M_1 f(x) > \left( \frac{\lambda}{4c_1} \right)^{n-1} \right\} \right| \label{LS_f}\\
		&+ \left| \left\{ x \in B\left( 0, \frac{1}{2} \right) : M\left[ |G|^\frac{n}{n-1}\right](x) > \left( \frac{\lambda}{4c_1} \right)^n \right\} \right| \label{LS_G}\\
		&+ \left| \left\{ x \in B\left( 0, \frac{1}{2} \right) : M[|\g u|^{n-\ve}](x)^\frac{1}{n-\ve} > \frac{\lambda}{ 4\Gamma \left( \frac{\sigma}{\tau^\frac{n}{n-1}} + \tau^\alpha \sigma \right) } \right\} \right| \label{LS_Mgu}\\
		&+ \left| \left\{ x \in B\left( 0, \frac{1}{2} \right) : M^\sharp_{n-\ve}(\g u)(x) > \frac{\lambda}{4 \Gamma \tau^\alpha} \right\} \right|. \label{LS_Msgu}
	\end{align}
	We define
	\begin{align}\label{def_lambda0}
		\lambda_0 = 8 \Gamma \tau^{\alpha - \frac{n}{n-\ve}} \left( \frac{2^n}{|B_1|} \right)^\frac{1}{n-\ve} \left( \int_{B(0,1)} |\g u|^{n-\ve} \right)^\frac{1}{n-\ve}.
	\end{align}
	We claim that if $\lambda > \lambda_0$, then we can rewrite \eqref{LS_Msgu}, namely
	\begin{align}\label{LS_Msgu1}
		\left| \left\{ x \in B\left( 0, \frac{1}{2} \right) : M^\sharp_{n-\ve}(\g u) > \frac{\lambda}{4\Gamma \tau^\alpha} \right\} \right| = \left| \left\{ x \in B\left( 0, \frac{1}{2} \right) : M^\sharp_{\tau,n-\ve}(\g u) > \frac{\lambda}{4 \Gamma \tau^\alpha} \right\} \right|, \quad \forall \lambda > \lambda_0.
	\end{align}
	Indeed, for any $x\in B(0,\frac{1}{2})$, it holds
	\begin{align*}
		M^\sharp_{n-\ve}(\g u)(x) = \max\left( M^\sharp_{\tau,n-\ve}(\g u)(x), \sup \left\{ \left( \mvint_{B(x,\rho)} |\g u - (\g u)_{B(x,\rho)}|^{n-\ve} \right)^\frac{1}{n-\ve} : B(x,\rho) \subset B(0,1),\, B\left( x,\frac{\rho}{\tau} \right) \not\subset B(0,1) \right\} \right).
	\end{align*}
	If $x \in B(0,1/2)$ and $B\left( x,\frac{\rho}{\tau} \right) \not\subset B(0,1)$, then
	$\rho \geq \frac{\tau}{2}$. Therefore,
	\begin{align*}
		\left( \mvint_{B(x,\rho)} |\g u - (\g u)_{B(x,\rho)}|^{n-\ve} \right)^\frac{1}{n-\ve} &\leq \left( \frac{2^n}{\tau^n |B_1|} \right)^\frac{1}{n-\ve} \left( \int_{B(0,1)} |\g u|^{n-\ve} \right)^\frac{1}{n-\ve}.
	\end{align*}
	If $\lambda>\lambda_0$, then
	\begin{align*}
		\frac{\lambda}{4\Gamma \tau^\alpha} &> 2\tau^{- \frac{n}{n-\ve}} \left( \frac{2^n}{|B_1|} \right)^\frac{1}{n-\ve} \left( \int_{B(0,1)} |\g u|^{n-\ve}  \right)^\frac{1}{n-\ve} \\
		&> 2\left( \frac{2^n}{\tau^n |B_1|} \right)^\frac{1}{n-\ve} \left( \int_{B(0,1)} |\g u|^{n-\ve} \right)^\frac{1}{n-\ve}.
	\end{align*}
	So \eqref{LS_Msgu1} holds for $\lambda>\lambda_0$.\\
	To estimate the right-hand side of \eqref{LS_f}, we use a Vitali covering. If $x\in B\left(0,\frac{1}{2}\right)$ satisfy $ M_1 f(x) > \left( \frac{\lambda}{4c_1} \right)^{n-1}$, there exists $r_x>0$ such that $B(x,r_x) \subset B(0,1)$ and
	\begin{align*}
		\frac{1}{r_x^{n-1}} \int_{B(x,r_x)} |f| >\left( \frac{\lambda}{4c_1} \right)^{n-1}.
	\end{align*}
	We can write this in the form
	\begin{align*}
		r_x^n < \left( \frac{4 c_1 }{\lambda} \right)^n \left( \int_{B(x,r_x)} |f| \right)^\frac{n}{n-1}.
	\end{align*}
	Now, we can cover the set $\left\{M_1 f > \left( \frac{\lambda}{4c_1} \right)^{n-1} \right\}$ by balls $(B(x_i,10r_{x_i}))_{i\in I}$ such that the balls $(B(x_i,r_{x_i}))_{i\in I}$ are disjoint. Then, \eqref{LS_f} can be estimated in the following manner :
	\begin{align}
		\left| \left\{ x \in B\left( 0, \frac{1}{2} \right) : M_1 f(x) > \left( \frac{\lambda}{4c_1} \right)^{n-1} \right\} \right| &\leq \sum_{i\in I} |B(x_i, 10 r_{x_i})| \nonumber\\
		&\leq 10^n |B_1| \sum_{i\in I} r_{x_i}^n \nonumber\\
		&\leq 10^n |B_1| \sum_{i\in I} \left( \frac{4 c_1 }{\lambda} \right)^n \left( \int_{B(x_i,r_{x_i})} |f| \right)^\frac{n}{n-1} \nonumber\\
		&\leq 40^n |B_1|\frac{c_1^n}{\lambda^n} \left( \int_{B(0,1)} |f| \right)^\frac{1}{n-1} \sum_{i\in I}\int_{B(x_i,r_{x_i})} |f| \nonumber\\
		&\leq 40^n |B_1|\frac{c_1^n}{\lambda^n} \left( \int_{B(0,1)} |f| \right)^\frac{n}{n-1}. \label{LS_f1}
	\end{align}
As well for \eqref{LS_G}. It holds
\begin{align*}
	\left| \left\{ x \in B\left( 0, \frac{1}{2} \right) : M\left[ |G|^\frac{n}{n-1}\right](x) > \left( \frac{\lambda}{4c_1} \right)^n \right\} \right| &= \left| \left\{x \in B\left( 0, \frac{1}{2} \right) : \exists r_x >0, \mvint_{B(x,r_x)} |G|^\frac{n}{n-1} > \left( \frac{\lambda}{4c_1} \right)^n \right\} \right|\\
	&= \left| \left\{x \in B\left( 0, \frac{1}{2} \right) : \exists r_x >0, \left( \frac{4c_1}{\lambda} \right)^n\int_{B(x,r_x)} |G|^\frac{n}{n-1} > r_x^n \right\} \right|\\
	&\leq 40^n |B_1|\frac{c_1^n}{\lambda^n} \int_{B(0,1)} |G|^\frac{n}{n-1}.
\end{align*}
	Now we estimate the $L^{(n,\infty)}$-norm of $M^\sharp_{\tau,n-\ve}(\g u)$ by separating the cases $\lambda\leq \lambda_0$ and $\lambda>\lambda_0$ :
	\begin{align}
		& \sup_{\lambda>0} \lambda^n \left| \left\{ x \in B\left( 0, \frac{1}{2} \right) : M^\sharp_{\tau,n-\ve}(\g u)(x) > \lambda \right\} \right| \nonumber\\
		& \leq \sup_{\lambda>\lambda_0} \lambda^n \left| \left\{ x \in B\left( 0, \frac{1}{2} \right) : M^\sharp_{\tau,n-\ve}(\g u)(x) > \lambda \right\} \right| \nonumber \\
		&+ \sup_{\lambda\leq \lambda_0} \lambda^n \left| \left\{ x \in B\left( 0, \frac{1}{2} \right) : M^\sharp_{\tau,n-\ve}(\g u)(x) > \lambda \right\} \right| \nonumber \\
		&\leq \sup_{\lambda>\lambda_0} \lambda^n \left| \left\{ x \in B\left( 0, \frac{1}{2} \right) : M^\sharp_{\tau,n-\ve}(\g u)(x) > \lambda \right\} \right| \label{Lninfty_MsGu1}\\
		&+ 2^{-n} |B_1| \lambda_0^n. \label{Lninfty_MsGu2}
	\end{align}
	We estimate \eqref{Lninfty_MsGu1} thanks to \eqref{LS_f1}-\eqref{LS_Mgu}-\eqref{LS_Msgu1} :
	\begin{align}
		& \sup_{\lambda>\lambda_0} \lambda^n \left| \left\{ x \in B\left( 0, \frac{1}{2} \right) : M^\sharp_{\tau,n-\ve}(\g u)(x) > \lambda \right\} \right| \nonumber \\
		&\leq 40^n |B_1| c_1^n \left( \int_{B(0,1)} |f| \right)^\frac{n}{n-1}  + 40^n |B_1| c_1^n \int_{B(0,1)} |G|^\frac{n}{n-1} \label{Lninfty_MsGu11}\\
		&+ 4^n \Gamma^n \left( \frac{\sigma}{\tau^\frac{n}{n-1}} + \tau^\alpha \sigma \right)^n \left\| M[|\g u|^{n-\ve}]^\frac{1}{n-\ve} \right\|_{L^{(n,\infty)}\left( B(0,\frac{1}{2}) \right)}^n \label{Lninfty_MsGu12}\\
		&+ 4^n \Gamma^n \tau^{n\alpha} \left\| M^\sharp_{\tau,n-\ve}(\g u) \right\|_{L^{(n,\infty)}\left( B(0,\frac{1}{2}) \right)}^n. \label{Lninfty_MsGu13}
	\end{align}
	Using \eqref{Lninfty_MsGu11}-\eqref{Lninfty_MsGu12}-\eqref{Lninfty_MsGu13} and \eqref{Lninfty_MsGu2}, we obtain
	\begin{align*}
		\left\| M^\sharp_{\tau,n-\ve}(\g u) \right\|_{L^{(n,\infty)}(B(0,1/2))}^n &\leq 40^n |B_1| c_1^n \left( \int_{B(0,1)} |f| \right)^\frac{n}{n-1}  + 40^n |B_1| c_1^n \int_{B(0,1)} |G|^\frac{n}{n-1}\\
		&+ 4^n \Gamma^n \left( \frac{\sigma}{\tau^\frac{n}{n-1}} + \tau^\alpha \sigma \right)^n \left\| M[|\g u|^{n-\ve}]^\frac{1}{n-\ve} \right\|_{L^{(n,\infty)}\left( B(0,\frac{1}{2}) \right)}^n \\
		&+ 4^n \Gamma^n \tau^{n\alpha} \left\| M^\sharp_{\tau,n-\ve}(\g u) \right\|_{L^{(n,\infty)} \left( B(0,\frac{1}{2}) \right)}^n \\
		&+ 2^{-n}|B_1| \lambda_0^n.
	\end{align*}
	Hence, if $\tau \leq \tau_1$, where
	\begin{align}\label{def_t1}
		\tau_1 := \min\left( (8 \Gamma)^{-1/\alpha}, \tau_0 \right).
	\end{align}
	then $4^n \Gamma^n \tau^{n\alpha} \leq 2^{-n\alpha}$ and we obtain, using the definition \eqref{def_lambda0} of $\lambda_0$ :
	\begin{align*}
		\left\| M^\sharp_{\tau,n-\ve}(\g u) \right\|_{L^{(n,\infty)}\left( B(0,\frac{1}{2}) \right)}^n &\leq \frac{40^n |B_1| c_1^n}{1-2^{-n\alpha}} \left( \int_{B(0,1)} |f| \right)^\frac{n}{n-1}+ \frac{40^n |B_1| c_1^n}{1-2^{-n\alpha}} \int_{B(0,1)} |G|^\frac{n}{n-1}\\
		&+ \frac{4^n }{1-2^{-n\alpha}} \Gamma^n \left( \frac{\sigma}{\tau^\frac{n}{n-1}} + \tau^\alpha \sigma \right)^n \left\| M[|\g u|^{n-\ve}]^\frac{1}{n-\ve} \right\|_{L^{(n,\infty)}\left( B(0,\frac{1}{2}) \right)}^n \\
		& + 4^n \Gamma^n \tau^{n\left( \alpha - \frac{n}{n-\ve} \right) } \left( \frac{2^n}{|B_1|} \right)^\frac{n}{n-\ve} \left( \int_{B(0,1)} |\g u|^{n-\ve} \right)^\frac{n}{n-\ve}.
	\end{align*}
	Since $\tau\in(0,1)$ and $\ve\in(0,1)$, it holds $\tau^{n\left( \alpha - \frac{n}{n-\ve} \right) } \leq \tau^{n\left( \alpha - \frac{n}{n-1} \right) }$ and $\left( \frac{2^n}{|B_1|} \right)^\frac{n}{n-\ve} \leq \max\left( \frac{2^n}{|B_1|}, \left( \frac{2^n}{|B_1|} \right)^\frac{n}{n-1} \right)$. We conclude thanks to the estimate
	\begin{align*}
		\left\| M[|\g u|^{n-\ve}]^\frac{1}{n-\ve} \right\|_{L^{(n,\infty)}\left( B(0,\frac{1}{2}) \right)} &\leq \left\| M[|\g u|^{n-\ve}] \right\|_{L^{(\frac{n}{n-\ve},\infty)}\left( B(0,\frac{1}{2}) \right)}^\frac{1}{n-\ve} \\
		&\leq c(n)\left\| |\g u|^{n-\ve} \right\|_{L^{(\frac{n}{n-\ve},\infty)}\left( B(0,1) \right)}^\frac{1}{n-\ve} \\
		&\leq c(n)\|\g u \|_{L^{(n,\infty)}\left( B(0,1) \right)}.
	\end{align*}
\end{proof}

To apply known results on the sharp maximal function, we show that in the definition of $M^\sharp_{\tau,n-\ve}$, one can replace the supremum over cubes and not over balls. We define
\begin{align*}
	M^\sharp_{\tau,n-\ve,c} g(x) &= \sup\left\{ \left( \mvint_{Q} |g - (g)_{Q}|^{n-\ve} \right)^\frac{1}{n-\ve} : \frac{1}{\tau}Q \subset B(0,1), x\in Q, Q \text{ is a cube} \right\},
\end{align*}
where $\frac{1}{\tau}Q$ is the cube of same center as $Q$, with side length $\frac{1}{\tau}$ times the side length of $Q$. We show that the maximal functions $M^\sharp_{\tau,n-\ve}$ and $M^\sharp_{\tau,n-\ve,c}$ are pointwise comparable.

\begin{lemma}\label{equivalence_ball_cube}
	There exists $C=C(n)>1$ and $a=a(n)>1$ such that for any function $g\in L^{n-\ve}_{loc}(B(0,1))$, $\tau\in(0,\frac{1}{a^2})$ and $x\in B(0,1)$ it holds
	\begin{align*}
		\frac{1}{C}M^\sharp_{\tau,n-\ve}g(x) \leq M^\sharp_{a\tau,n-\ve,c}g(x) \leq C M^\sharp_{a^2\tau,n-\ve}g(x).
	\end{align*}
\end{lemma}

\begin{proof}
	Given a ball $B\subset B(0,1)$ such that $\frac{1}{\tau} B \subset B(0,1)$, consider $Q$ the smallest cube having the same center as $B$ and containing $B$. There exists $C=C(n)$ such that $1\leq \frac{|Q|}{|B|} \leq C$. In particular, there exists $a=a(n)>1$ such that $\frac{1}{a\tau} Q \subset B(0,1)$. Then, for any function $g\in C^\infty(B(0,1))$, it holds
	\begin{align*}
		\left( \mvint_B |g-(g)_B|^{n-\ve} \right)^\frac{1}{n-\ve} &\leq \left( \frac{|Q|}{|B|} \mvint_Q |g - (g)_B|^{n-\ve} \right)^\frac{1}{n-\ve} \\
		&\leq \left( \frac{|Q|}{|B|} \right)^\frac{1}{n-\ve} \left[ \left( \mvint_Q |g-(g)_Q|^{n-\ve} \right)^\frac{1}{n-\ve} + |(g)_Q - (g)_B| \right] \\
		&\leq \left( \frac{|Q|}{|B|} \right)^\frac{1}{n-\ve} \left[ \left( \mvint_Q |g-(g)_Q|^{n-\ve} \right)^\frac{1}{n-\ve} + \mvint_B |(g)_Q - g| \right] \\
		&\leq \left( \frac{|Q|}{|B|} \right)^\frac{1}{n-\ve} \left[ \left( \mvint_Q |g-(g)_Q|^{n-\ve} \right)^\frac{1}{n-\ve} + \frac{|Q|}{|B|} \mvint_Q |(g)_Q - g| \right] \\
		&\leq \left( \frac{|Q|}{|B|} \right)^\frac{1}{n-\ve} \left( 1+  \frac{|Q|}{|B|} \right) \left( \mvint_Q |g-(g)_Q|^{n-\ve} \right)^\frac{1}{n-\ve}.
	\end{align*}
	Using $\ve\in(0,1)$, we conclude that
	\begin{align*}
		\left( \mvint_B |g-(g)_B|^{n-\ve} \right)^\frac{1}{n-\ve} &\leq \left( \frac{|Q|}{|B|} \right)^\frac{1}{n-1} \left( 1+  \frac{|Q|}{|B|} \right) \left( \mvint_Q |g-(g)_Q|^{n-\ve} \right)^\frac{1}{n-\ve}
	\end{align*}
	Therefore, $M^\sharp_{\tau,n-\ve}g(x) \leq C(n) M^\sharp_{a\tau,n-\ve,c}g(x)$ for some $a=a(n)>1$. The converse is shown with in the same manner: given a cube $Q$, we consider the smallest ball $B$ containing $Q$ with the same center. The computations are similar.
\end{proof}

We now obtain a lower bound for $\|M^\sharp_{a\tau,1,c} g\|_{L^{p,q}(B(0,1))}$, given $g \in L^\infty(B(0,1))$. This is a direct adaptation of the proof of \cite[Corollary 7.5, p.380]{bennett1988}.

\begin{lemma}\label{goingtoLpq}
	Let $a$ be defined in \Cref{equivalence_ball_cube} and $\tau\in(0,\tau_1)$, where $\tau_1$ is defined in \Cref{estimate_norm_MsGu}. Let $p\in(1,\infty)$ and $q\in[1,\infty]$. For any function $g\in L^{(p,q)}(B(0,1))$, it holds
	\begin{align*}
		\|g\|_{L^{p,q}(B(0,a\tau))} &\leq C(n,p,q)\left( \|M^\sharp_{a\tau,1,c} g\|_{L^{p,q}(B(0,1))} + \mvint_{B(0,a\tau)} |g| \right).
	\end{align*}
\end{lemma}

\begin{proof}
	We start with the straightforward inequality
	\begin{align*}
		\|M^\sharp_{a\tau,1,c} g\|_{L^{p,q}(B(0,1))} \geq \|\tilde{M}^\sharp \tilde{g}\|_{L^{p,q}(Q_0)},
	\end{align*}
	where $Q_0 := [-\frac{a}{2}\tau , \frac{a}{2} \tau]^n$, $\tilde{g} := g\mathbf{1}_{Q_0}$ and
	\begin{align*}
		\tilde{M}^\sharp \tilde{g}(x) &= \sup\left\{  \mvint_{Q} |\tilde{g} - (\tilde{g})_{Q}| : Q \subset Q_0, Q \ni x, Q\text{ is a cube} \right\}.
	\end{align*}
	Thanks to \Cref{th:Msharp_star}, it holds
	\begin{align*}
		\tilde{g}^{**}(t) -\tilde{g}^*(t) &\leq c(n) (\tilde{M}^\sharp \tilde{g})^*(t).
	\end{align*}
	By integration by parts, for any $0<t\leq u \leq \frac{1}{6}|Q_0|$, it holds
	\begin{align*}
		\tilde{g}^{**}(t) - \tilde{g}^{**}(u) &= \int_t^u \left( \tilde{g}^{**}(s) - \tilde{g}^*(s) \right) \frac{ds}{s}.
	\end{align*}
	The choice $ u = \frac{1}{6}|Q_0|$ leads to
	\begin{align*}
		\tilde{g}^{**}(t) &\leq \tilde{g}^{**}\left( \frac{|Q_0|}{6} \right) + c(n)\int_t^\infty(\tilde{M}^\sharp \tilde{g})^*(s) \frac{ds}{s} \\
		&\leq c(n)\mvint_{Q_0} |\tilde{g}| + c(n)\int_t^\infty(\tilde{M}^\sharp \tilde{g})^*(s) \frac{ds}{s}.
	\end{align*}
	Thanks to Hardy's inequality, see \cite[(3.19), Lemma 3.9, p.124]{bennett1988} :
	\begin{align*}
		\|\tilde{g}\|_{L^{p,q}(Q_0)} &\leq c(n,p,q)|Q_0|^\frac{1}{p}\mvint_{Q_0} |\tilde{g}| + c(n,p,q)\|\tilde{M}^\sharp \tilde{g} \|_{L^{p,q}(Q_0)}.
	\end{align*}
	Since $Q_0 \subset B(0,1)$, it holds $|Q_0|^\frac{1}{p} \leq c(n,p)$.
\end{proof}

Combining \Cref{estimate_norm_MsGu}, \Cref{equivalence_ball_cube}, and \Cref{goingtoLpq} we readily obtain \Cref{Lninfty_estimate}.

\section{Regularity iteration: Proof of Theorem~\ref{th:regularityLn1}}\label{s:iteration}
The goal of this section is to prove the following result.
\begin{proposition}\label{pr:iteration}
	There exists $\gamma_0,\lambda,\nu \in(0,1)$ and $k>0$ depending on $n$ and $N$ such that the following holds.
	Consider $u \in W^{1,n}(B^n(0,1);\R^N)$, $Q\in W^{1,n}(B^n(0,1);SO(N))$ and $\Omega^Q\in L^{(n,1)}(B^n(0,1);\R^n \otimes \R^{N\times N})$ such that
	\begin{align*}
		-\di(|\g u|^{n-2} Q\g u) = \Omega^Q |\g u|^{n-2} \g u \quad \text{in $B(0,1)$}
	\end{align*}
	Assume that
	\begin{align*}
		\|\g Q\|_{L^n(B(0,1))} + \|\Omega^Q \|_{L^{(n,1)}(B(0,1))} \leq \gamma_0.
	\end{align*}
	Then for any ball $B(x,r)\subset B(0,1)$,
	\begin{align*}
		\|\g u\|_{L^{(n,\infty)}(B(x,\lambda r))}^{n-\ve} + \frac{k}{(\lambda r)^\ve} \int_{B(x,\lambda r)} |\g u|^{n-\ve} &\leq \nu \left[ \|\g u\|_{L^{(n,\infty)}(B(x,r))}^{n-\ve} + \frac{k}{r^\ve} \int_{B(x,r)} |\g u|^{n-\ve} \right].
	\end{align*}
\end{proposition}

Once \Cref{pr:iteration} is proven, we can conclude
\begin{proof}[Proof of \Cref{th:regularityLn1}]
	By a standard iterative argument we obtain from \Cref{pr:iteration} the existence of $C>0$ and $\beta\in(0,1)$ such that for any ball $B(x,r)\subset B(0,1)$:
	\begin{align*}
		\|\g u\|_{L^{(n,\infty)}(B(x,r))}^{n-\ve} + \frac{k}{r^\ve} \int_{B(x, r)} |\g u|^{n-\ve} &\leq Cr^\beta.
	\end{align*}
	By \cite[Corollary 3.2]{han2011}, we conclude that $u$ is continuous. 
\end{proof}

\begin{proof}[Proof of \Cref{pr:iteration}]
	For simplicity of presentation we assume that $x=0$ and $r=1$.
	
	We fix $\sigma,\tau \in(0,1)$ to be specified later.

	First we have the estimate coming from the \Cref{Lninfty_estimate}. We obtain $\ve_1(\sigma)\in(0,1)$ and a constant $\gamma_1(\sigma)\in(0,1)$ such that if $\|\g Q\|_{L^n(B(0,1))} \leq \gamma_1$, then the following holds. There exists $\Gamma=\Gamma(n,N)>0$, $\alpha=\alpha(n,N)\in(0,1)$ and $\tau_1 =\tau_1(n,N)\in(0,1)$ such that for
	\[
	\tau \in (0,\tau_1), \quad  \eps \in (0,\eps_0)
	\]
	\begin{align*}
		\| \g u \|_{L^{(n,\infty)}(B(0,\tau ))}^{n-\ve} &\leq c_1(\sigma,\tau,\eps) \left\| \Omega |\g u|^{n-1} \right\|_{L^1(B(0,1))}^\frac{n-\ve}{n-1} + c_2(\tau)\int_{B(0,1)}|\g u|^{n-\ve} \\
		&+ \Gamma\left( \frac{\sigma}{\tau^\frac{n}{n-1}} + \tau^\alpha \right)^{n-\ve} \|\g u\|_{L^{(n,\infty)}(B(0,1))}^{n-\ve}.
	\end{align*}
	Then we have the estimate from \Cref{la:easyiwaniec}: there exists $\ve_0(\sigma,\tau)\in(0,1)$ and $\gamma_0(\sigma)\in(0,1)$ such that if $\|\g Q\|_{L^n(B(0,1))}\leq \gamma_0$, then for the choice $\ve\in (0,\ve_0)$, it holds
	\begin{align*}
		\frac{1}{\tau^\ve} \int_{B(0,\tau )} |\g u|^{n-\ve} &\leq  \sigma \int_{B(0,1)} |\g u|^{n-\ve} + c_3(\sigma)\left( \int_{B(0,1)}|\g u|^{n-\ve} - \frac{1}{\tau^\ve} \int_{B(0,\tau)} |\g u|^{n-\ve} \right) \\
		&+ c_4(\sigma,\tau,\eps) \left\| \Omega |\g u|^{n-1} \right\|_{L^1(B(0,1))}^\frac{n-\ve}{n-1}.
	\end{align*}
	
	From now on we set $\eps = \min\{\eps_0,\eps_1\}$ (which depends on $\tau$ and $\sigma$).
	
	Thanks to Hölder's inequality
	\[
	\left\| \Omega |\g u|^{n-1} \right\|_{L^1(B(0,1))}\leq \|\Omega\|_{L^{(n,1)}(B(0,1))}\|\g u\|_{L^{(n,\infty)}(B(0,1))}^{n-1}.
	\]
	
	Plugging this into the above estimates, we obtain first
	\begin{align}\label{estimate1}
		\| \g u \|_{L^{(n,\infty)}(B(0,\tau ))}^{n-\ve} &\leq \left( c_1(\sigma,\tau,\eps) \|\Omega\|_{L^{(n,1)}(B(0,1))}^\frac{n-\ve}{n-1} + \Gamma\left( \frac{\sigma}{\tau^\frac{n}{n-1}} + \tau^\alpha \right)^{n-\ve}  \right) \|\g u \|_{L^{(n,\infty)}(B(0,1))}^{n-\ve} \\
		&+ c_2(\tau)\int_{B(0,1)}|\g u|^{n-\ve}, \nonumber
	\end{align}
	then
	\begin{align}\label{estimate2}
		\frac{1}{\tau^\ve} \int_{B(0,\tau )} |\g u|^{n-\ve} &\leq  \left( \sigma +c_3(\sigma) \right)\int_{B(0,1)} |\g u|^{n-\ve} - \frac{c_3(\sigma)}{\tau^\ve} \int_{B(0,\tau)} |\g u|^{n-\ve}  \\
		&+ c_4(\sigma,\tau,\eps) \|\Omega\|_{L^{(n,1)}(B(0,1))}^\frac{n-\ve}{n-1}\|\g u\|_{L^{(n,\infty)}(B(0,1))}^{n-\ve}. \nonumber
	\end{align}
	Adding \eqref{estimate1} and $10c_2(\tau)$ times the estimate \eqref{estimate2}, we obtain
	\begin{align*}
		& \| \g u \|_{L^{(n,\infty)}(B(0,\tau ))}^{n-\ve} + \frac{10c_2(\tau)}{\tau^\ve}\int_{B(0,\tau )} |\g u|^{n-\ve} \\
		&\leq \left[ \Big( c_1(\sigma,\tau,\eps) + 10c_2(\tau)c_4(\sigma,\tau,\eps)\Big) \|\Omega\|_{L^{(n,1)}(B(0,1))}^\frac{n-\ve}{n-1} + \Gamma\left( \frac{\sigma}{\tau^\frac{n}{n-1}} + \tau^\alpha \right)^{n-\ve} \right] \|\g u\|_{L^{(n,\infty)}(B(0,1))}^{n-\ve} \\
		&+ \Big( c_2(\tau) + 10c_2(\tau)\sigma +10c_2(\tau) c_3(\sigma) \Big) \int_{B(0,1)} |\g u|^{n-\ve}  - \frac{10c_2(\tau) c_3(\sigma)}{\tau^\ve} \int_{B(0,\tau)} |\g u|^{n-\ve}.
	\end{align*}
	The last term goes on the left-hand side,
	\begin{align}\label{iteration1}
		&\| \g u \|_{L^{(n,\infty)}(B(0,\tau ))}^{n-\ve} + \frac{10c_2(\tau)(1+c_3(\sigma))}{\tau^\ve}\int_{B(0,\tau )} |\g u|^{n-\ve} \\
		&\leq \left[ \Big( c_1(\sigma,\tau,\eps) + 10c_4(\sigma,\tau,\eps)\Big) \|\Omega\|_{L^{(n,1)}(B(0,1))}^\frac{n-\ve}{n-1} + \Gamma\left( \frac{\sigma}{\tau^\frac{n}{n-1}} + \tau^\alpha \right)^{n-\ve} \right] \|\g u\|_{L^{(n,\infty)}(B(0,1))}^{n-\ve} \nonumber\\
		&+ c_2(\tau)\Big( 1 + 10 \sigma + 10c_3(\sigma) \Big) \int_{B(0,1)} |\g u|^{n-\ve}. \nonumber
	\end{align}
	We bound the term $\left( \frac{\sigma}{\tau^\frac{n}{n-1}} + \tau^\alpha \right)^{n-\ve}$ independantly of $\ve$ in the following way,
	\begin{align*}
		\left( \frac{\sigma}{\tau^\frac{n}{n-1}} + \tau^\alpha \right)^{n-\ve} &\leq 2^{n-\ve} \left( \frac{\sigma^{n-\ve}}{\tau^{n\frac{n-\ve}{n-1}} } + \tau^{\alpha(n-\ve)} \right) \\
		&\leq 2^n \left( \frac{\sigma^{n-1} }{ \tau^\frac{n^2}{n-1} } + \tau^{\alpha(n-1)} \right).
	\end{align*}
	Using this inequality in \eqref{iteration1}, we obtain
	\begin{align}\label{iteration2}
		&\| \g u \|_{L^{(n,\infty)}(B(0,\tau ))}^{n-\ve} + \frac{10c_2(\tau)(1+c_3(\sigma))}{\tau^\ve}\int_{B(0,\tau )} |\g u|^{n-\ve} \\
		&\leq \left[ \Big( c_1(\sigma,\tau,\eps) + 10c_4(\sigma,\tau,\eps)\Big) \|\Omega\|_{L^{(n,1)}(B(0,1))}^\frac{n-\ve}{n-1} + 2^n \Gamma \left( \frac{\sigma^{n-1} }{ \tau^\frac{n^2}{n-1} } + \tau^{\alpha(n-1)} \right) \right] \|\g u\|_{L^{(n,\infty)}(B(0,1))}^{n-\ve} \nonumber \\
		&+ c_2(\tau)\Big( 1 + 10 \sigma + 10c_3(\sigma) \Big) \int_{B(0,1)} |\g u|^{n-\ve}. \nonumber
	\end{align}
	Now we choose the parameters. We consider first $\tau < \tau_1$ such that
	\begin{align*}
		2^n \Gamma \tau^{\alpha(n-1)} \leq \frac{1}{10^{1000}}.
	\end{align*}
	Then we choose $\sigma$ small enough to obtain
	\begin{align*}
		\sigma < \frac{1}{10^{1000}}, &  & 2^n \Gamma \frac{\sigma^{n-1}}{\tau^\frac{n^2}{n-1}} \leq \frac{1}{10^{1000}}.
	\end{align*}
	Finally we choose $\|\Omega\|_{L^{(n,1)}(B(0,1))}\leq \gamma_2$, where $\gamma_2=\gamma_2(n,N)\in(0,1)$ satisfies
	\begin{align*}
		\Big( c_1(\sigma,\tau,\eps) + 10c_4(\sigma,\tau,\eps)\Big) \gamma_2 \leq \frac{1}{10^{1000}}.
	\end{align*}
	Let
	\begin{align*}
		k &:= 10c_2(\tau)(1+c_3(\sigma)),\\
		\nu &:= \frac{ 1 + \frac{1}{10^{999}} + 10c_3(\sigma) }{ 10(1+c_3(\sigma)) } \in(0,1).
	\end{align*}
	From \eqref{iteration2}, we obtain
	\begin{align*}
		&\| \g u \|_{L^{(n,\infty)}(B(0,\tau ))}^{n-\ve} + \frac{k}{\tau^\ve}\int_{B(0,\tau )} |\g u|^{n-\ve} \\
		&\leq \frac{3}{10^{1000}}  \|\g u\|_{L^{(n,\infty)}(B(0,1))}^{n-\ve} + c_2(\tau)\left( 1 + \frac{1}{10^{999}} + 10c_3(\sigma) \right) \int_{B(0,1)} |\g u|^{n-\ve} \\
		&\leq \nu \left[ \frac{30(1+c_3(\sigma))}{10^{1000} \left( 1 + \frac{1}{10^{999}} + 10c_3(\sigma) \right) }\|\g u\|_{L^{(n,\infty)}(B(0,1))}^{n-\ve} + k \int_{B(0,1)} |\g u|^{n-\ve} \right] \\
		&\leq \nu \left[ \|\g u\|_{L^{(n,\infty)}(B(0,1))}^{n-\ve} + 10c_2(\tau)(1+c_3(\sigma)) \int_{B(0,1)} |\g u|^{n-\ve} \right].
	\end{align*}
\end{proof}

\section{Applications to harmonic maps: Proof of Corollary~\ref{co:harmonicmap}}\label{s:applications}
Here we discuss applications of the \Cref{th:regularityLn1} to $n$-harmonic maps into a manifold.

Let $(\Nr,h)$ be a smooth Riemannian manifold. For any $u \in W^{1,n}(B(0,1);\Nr)$, we let
the Dirichlet energy be
\begin{align*}
	D(u) := \int_{B(0,1)} |du|^n_h.
\end{align*}
Thanks to Nash embedding theorem, we can consider $\Nr$ as a smooth submanifold of $\R^N$, for some $N\geq 1$.

\begin{theorem}
	Let $\Nr$ be a smooth submanifold of $\R^N$ with second fundamental form $A$ satisfying $A\in W^{1,\infty}(\Nr)$. Consider $u \in W^{1,(n,2)}(B^n(0,1);N)$ a critical point of $D$. Then $u$ is $C^{1,\alpha}_{loc}$.
\end{theorem}

\begin{proof}
	We follow Rivi\`ere's argument in \cite{R07}. The Euler-Lagrange equation is given by
	\begin{align*}
		-\lap_n u &= |\g u|^{n-2} A(u)(\g u, \g u).
	\end{align*}
	Componentwise, this system is understood as
	\begin{align*}
		\forall i\in \{1,\ldots,N\}, \ \ \ \ \ -\di\left( |\g u|^{n-2} \g u^i \right) &= |\g u|^{n-2} A^i_{jk}(u) \scal{\g u^j }{\g u^k}.
	\end{align*}
	The second fundamental form takes values into $(T\Nr)^\perp$, so we have the orthogonality relation $A^i_{jk} \g u_i = 0$, for any $j,k\in\{1,\ldots,N\}$. Consequently, for any $i\in\{1,\ldots,N\}$, it holds
	\begin{align*}
		-\di\left( |\g u|^{n-2} \g u^i \right) &= |\g u|^{n-2} \scal{ A^i_{jk}(u) \g u^k - A^j_{ik}(u) \g u^k }{\g u^j}.
	\end{align*}
	If we set $\Omega_{ij} := A^i_{jk}(u) du^k - A^j_{ik}(u) du^k$, we obtain an skew-symmetric 1-form such that its exterior derivative is given by
	\begin{align*}
		d\Omega_{ij} &= \left[ (\dr_\alpha A^i_{jk})(u)- (\dr_\alpha A^j_{ik})(u) \right] du^\alpha\wedge du^k.
	\end{align*}
	By assumption the assumption $\g A\in L^\infty$ and $\g u \in L^{(n,2)}$, we obtain $d\Omega_{ij} \in L^{\left( \frac{n}{2}, 1 \right)}(B(0,1))$. Thanks to \Cref{th:main1}, we obtain that $u$ is continuous. The smallness assumption is satisfied on small balls by absolute continuity of the $L^{(n,2)}$-norm, see e.g. \cite[Theorem 8.5.1.]{PKJF2013}. Following the proof of \cite[Theorem 3.1]{hardt1987}, we obtain the $C^{1,\alpha}$-regularity.
\end{proof}

\section{Applications to H-System: Proof of Corollary~\ref{co:Hsystem}}\label{s:proofHsystem}

The $H$-systems are a closely related problem to $n$-harmonic maps. Instead of looking at arbitrary critical points of the Dirichlet energy for maps from $B^n(0,1)$ into $\R^{n+1}$, we consider the additional constraint of fixing the volume
\begin{align*}
	V(u) &= \int_{B(0,1)} \scal{u}{\dr_1 u\times \dr_2 u\times \cdots \times \dr_n u }.
\end{align*}
Under the constraint $V(u)=V_0$, the Euler-Lagrange equation is given by 
\begin{align*}
	-\lap_n u &= H \dr_1 u\times \dr_2 u\times \cdots \times \dr_n u,
\end{align*}
where $H$ is a constant depending on $V_0$.

\begin{theorem}
	Let $H\in W^{1,\infty}(\R^{n+1};\R)$. Assume $u\in W^{1,(n,\frac{n}{n-1})}(B^n(0,1);\R^{n+1})$ satisfies
	\begin{align*}
		-\lap_n u &= H(u) \dr_1 u\times \dr_2 u\times \cdots \times \dr_n u.
	\end{align*}
	Then $u$ is continuous.
\end{theorem}

\begin{proof}
The right-hand side of the $H$-system can be written as
\[
H(u)  B(u)^i \g u_i,
\]
where $B(u)$ satisfies :
\begin{itemize}
	\item a pointwise bound : $|B(u)|\aleq |\g u|^{n-1}$,
	\item a pointwise identity : for any $i\in\inter{1}{n+1}$, if $M^i := (\g u^1 ,\cdots ,\g u^{i-1},\g u^{i+1},\cdots,\g u^{n+1})$ then 
	\begin{align*}
		M^i B(u)^i = (\det M^i) I_n.
	\end{align*}
\end{itemize}
By \cite[Lemma 1.9]{bojarski1983}, $B(u)$ is divergence free, so $\Rz \cdot B(u)=0$. \\

Consider a cut-off function $\eta\in C^\infty_c(B(1))$ such that $\eta = 1$ in $B(\frac{1}{2})$, $|\eta|\leq 1$ and $|\g \eta|\aleq 1$. Let 
\begin{align*}
	\tilde{u} &= \eta\left( u - (u)_{B(\frac{1}{2})} \right) +  (u)_{B(\frac{1}{2})}.
\end{align*}
We decompose the following quantity on $\R^n$ :
\[
H(\tilde{u}) \nabla \tilde{u}\ B(\tilde{u}) = \underbrace{\underbrace{-[\Rz,H(\tilde{u})](\laps{1} \tilde{u})}_{=\Omega_1}\, B(\tilde{u})}_{=: f} + \underbrace{\Rz(H(\tilde{u}) \laps{1} \tilde{u})\, B(\tilde{u})}_{=:g}
\]
For the first term we observe by \Cref{th:crwlorentz} :
\[
\|\Omega_1 \|_{L^{(n,1)}(\R^n)} \aleq \|\nabla H(\tilde{u})\|_{L^{(n,2)}(\R^n)} \|\nabla \tilde{u}\|_{L^{(n,2)}(\R^n)} \aleq \|\nabla u\|_{L^{(n,2)}(B(0,1))}^2.
\]
We estimate $f$ by Hölder inequality :
\[
\|f\|_{L^1(\R^n)} \aleq \|\nabla u\|_{L^{(n,2)}(B(0,1))}^2 \|\nabla u\|_{L^{(n,\infty)}(B(0,1))}^{n-1}.
\]
We want to write $g$ as a divergence term. We estimate it by duality. For any $\vp\in C^\infty_c(\R^n)$, it holds
\begin{align*}
	\int_{\R^n} \Rz(H(\tilde{u}) \laps{1} \tilde{u})\, B(\tilde{u}) \varphi &= \int_{\R^n} [\Rz,\vp](H(\tilde{u}) \laps{1} \tilde{u})\, B(\tilde{u})\\
	&\aleq \| B(\tilde{u}) \|_{L^{(\frac{n}{n-1},\infty)}(\R^n)} \left\| [\Rz,\vp](H(\tilde{u}) \laps{1} \tilde{u}) \right\|_{L^{(n,1)}(\R^n)}\\
	&\aleq \|\nabla \tilde{u}\|_{L^{(n,\infty)}(\R^n)}^{n-1} \|\g \vp\|_{L^n(\R^n)} \|\nabla \tilde{u}\|_{L^{(n,\frac{n}{n-1})}(\R^n) }.
\end{align*}
By duality, there exists $G \in L^\frac{n}{n-1}(\R^n)$ such that
\[
g = \div G,
\]
with the estimate
\begin{align*}
	\|G\|_{L^{\frac{n}{n-1}}(\R^n)} &\aleq \|\nabla \tilde{u}\|_{L^{(n,\infty)}(\R^n)}^{n-1} \|\nabla \tilde{u}\|_{L^{(n,\frac{n}{n-1})}(\R^n) } \\
	&\aleq \|\nabla u\|_{L^{(n,\infty)}(B(1))}^{n-1} \|\nabla u\|_{L^{(n,\frac{n}{n-1})}(B(1)) }.	
\end{align*}
Namely we may take $G = -\Rz \brac{\lapms{1} g}$.

We observe that w.l.o.g. we may assume $\|\nabla u\|_{L^{(n,\frac{n}{n-1})}(B(1)) } \ll 1$, by absolute continuity of the $L^{(n,2)}$-norm, see e.g. \cite[Theorem 8.5.1.]{PKJF2013}.

Fix $\sigma\in(0,1)$ and $\tau \in (0,\tau_1)$. Thanks to \Cref{la:easyiwaniec}, we obtain the estimate
\begin{align*}
	\tau^{-\eps} \|\g u \|_{L^{n-\eps}(B(\frac{1}{4}))}^{n-\eps} =&\tau^{-\eps} \|\g \tilde{u} \|_{L^{n-\eps}(B(\frac{1}{4}))}^{n-\eps} \\
	\aleq &\left( \sigma + C_1(\sigma) \right) \|\g u\|_{L^{n-\eps}(B(1))}^{n-\eps} - C_1(\sigma) \tau^{-\eps}\|\g u \|_{L^{n-\eps}(B(1))}^{n-\eps}\\
	&+ C_2(\eps,\sigma) \|\g u\|_{L^{(n,\frac{n}{n-1})}(B(1))}^\frac{n-\eps}{n-1} \|\g u \|_{L^{(n,\infty)}(B(1))}^{n-\eps}.
\end{align*}
Thanks to \Cref{Lninfty_estimate}, it holds
\begin{align*}
	\|\g u \|_{L^{(n,\infty)}(B(\frac{1}{4}))} =& \|\g \tilde{u} \|_{L^{(n,\infty)}(B(\frac{1}{4}))} \\
	\aleq & C_3(\sigma,\eps,\tau)\|\g u\|_{L^{(n,\frac{n}{n-1})}(B(1))}^\frac{1}{n-1} \|\g u \|_{L^{(n,\infty)}(B(1))} + C_4(\tau)\|\g u \|_{L^{n-\eps}(B(1))}\\
	&+ \Gamma\left( \frac{\sigma}{\tau^\frac{n}{n-1}} + \tau^\alpha\right) \|\g u\|_{L^{(n,\infty)}(B(1))}.
\end{align*}
The rest of the proof goes like the proof of \Cref{pr:iteration}.
\end{proof}


\bibliographystyle{abbrv}%
\bibliography{bib}%

\end{document}